\documentclass[11pt,noindent]{amsart}
\usepackage{amssymb,amsmath,amsthm}
\usepackage{color}
\usepackage{graphicx}
\usepackage{epstopdf}
\usepackage{geometry}
\usepackage{hyperref}
\hypersetup{colorlinks=true,linkcolor=blue}

\usepackage{mathrsfs}

\definecolor{gray}{rgb}{0.5,0.5,0.7}
\definecolor{white}{rgb}{1,1,1}
\definecolor{black}{rgb}{0,0,0}
\definecolor{dgreen}{rgb}{0,0.65,0}
\definecolor{pink}{rgb}{0.7,0.5,0.5}
\definecolor{redred}{rgb}{1,0,0.5}
\definecolor{red}{rgb}{1,0,0}
\definecolor{rose}{rgb}{1,0.5,0.5}
\definecolor{red1}{rgb}{.8,0,0}
\definecolor{myst}{rgb}{1,0.894,0.882}
\definecolor{azure}{rgb}{0.941,1,1}
\definecolor{magenta}{rgb}{1,0,1}
\definecolor{dblue}{rgb}{0,0,0.7}
\definecolor{dcyan}{rgb}{0,0.7,0.7}
\definecolor{dmagenta}{rgb}{0.8,0,0.8}
\definecolor{fymaroyalblue}{rgb}{0.06,0.25,0.41}

\definecolor{giallo}{rgb}{0.8,0.8,0}
\definecolor{verde}{rgb}{0,0.4,0.4}

\newcommand{\N}{{\mathbb N}}

\newcommand{\R}{{\mathbb R}}

\newcommand{\F}{{\mathcal F}}

\newcommand{\vertiii}[1]{{\left\vert\kern-0.25ex\left\vert\kern-0.25ex\left\vert #1
		\right\vert\kern-0.25ex\right\vert\kern-0.25ex\right\vert}}

\def\0  {\mathbb{O}}

\def\A  {\mathcal{A}}

\def\R  {\mathbb{R}}

\def\and {\quad \text{and} \quad}

\def\F  {{\mathcal F}}

\newcommand{\e}{\quad \text{and} \quad}

\theoremstyle{definition}

\newtheorem*{problemPa}{Problem $\boldsymbol{P_\alpha}$}

\newtheorem*{P2'}{Problem $\boldsymbol{2'}$}

\newtheorem{theorem}{Theorem}[section]
\newtheorem{proposition}[theorem]{Proposition}

\newtheorem{lemma}[theorem]{Lemma}
\newtheorem{definition}[theorem]{Definition}
\newtheorem{remark}[theorem]{Remark}

\numberwithin{equation}{section}
\numberwithin{theorem}{section}

\definecolor{giallo}{rgb}{1,0.7,0}
\definecolor{verdone}{rgb}{0,0.5,0.2}
\definecolor{arancione}{rgb}{1,0.35,0}

\newcommand{\Dta}{\partial_t^\alpha}

\newcommand{\de}{\mathrm {d}}

\date{\today
	}

\title[]{\bf Identification problems for anisotropic time-fractional subdiffusion equations}

\begin{document}

	\maketitle
	

    	\centerline{\scshape Simone Creo and Maria Rosaria Lancia}
	{\footnotesize
		 \centerline{Department of Basic and Applied Sciences for Engineering, Sapienza Universit\`a di Roma}
		  \centerline{Via A. Scarpa 10, 00161 Roma, Italy}
        \centerline{ORCID IDs: 0000-0002-2083-2344, 0000-0002-6384-8311}
	}
	\smallskip

	\centerline{\scshape Andrea Mola}
	{\footnotesize
		 \centerline{Multi-scale Analysis of Materials Research Unit, IMT School for Advanced Studies Lucca}
		  \centerline{Piazza S. Francesco, 19, 55100 Lucca, Italy}
        \centerline{ORCID ID: 0000-0002-4483-0212}
	}
    	\smallskip

	\centerline{\scshape Gianluca Mola}
	
	{\footnotesize
		 \centerline{Department of Science and Engeneering, Sorbonne University at Abu Dhabi}
		  \centerline{Al Reem Island, 38044 Abu Dhabi, UAE}
          \centerline{ORCID ID: 0000-0003-1585-372X}
	}
\smallskip
    	\centerline{\scshape Silvia Romanelli}
	
	{\footnotesize
		 \centerline{Dipartimento di Matematica, Universit\`a  degli Studi di Bari Aldo Moro}
		  \centerline{Via E. Orabona 4, 70125 Bari, Italy}
        \centerline{ORCID ID: 0000-0001-5408-042X}
	}
	

	

	\begin{abstract}
		We investigate the inverse problem consisting in the identification of constant coefficients for a fractional-in-time partial differential equation governed by a finite sum of positive self-adjoint operators on a Hilbert space under energy-type overdeterminating conditions. We prove the uniqueness of the solution to the inverse problem when the fractional order $\alpha$ of the derivative is in $(0,1)$. A conditioned existence result is also provided, complemented with a suitable selection of numerical calculations. In addition, we prove that, as $\alpha\to 1^{-}$, the solution corresponding to $\alpha$ tends to the classical one ($\alpha=1$). Applications to examples of heat diffusion and elasticity are presented.
	\end{abstract}


		\section{Introduction}\label{introduction}

We consider the fractional-in-time equation
\begin{equation}\label{1}
	\partial_t^{\alpha} u(t) + \lambda_1 A_1\, u(t) + \cdots + \lambda_n A_n\, u(t) = 0, \quad t > 0,
\end{equation}
which describes the evolution of a concentration variable $u = u(t)$ on a Hilbert space $H$, where $\partial_t^{\alpha}$ denotes the Caputo-type  derivative of order $\alpha \in (0,1)$ (see Subsection \ref{fractional} for the definition). 
Under positivity conditions on the linear (possibly unbounded) operators $A_i: D(A_i) \to H$, the equation \eqref{1} describes models of subdiffusion (cf. \cite[Appendix C.2]{GW}), which include applications to conductors with fading memory. The sum of multiple operators allows to display effects related to \emph{anisotropic} subdiffusion along the $n$ directions implicitly delivered by $A_1, \dots, A_n$, each one weighted by its coefficient $\lambda_i$, for $i=1,\dots,n$, assumed from now on to be a positive constant.

In the limiting value $\alpha = 1$, the equation \eqref{1} accounts for the original diffusion equation 
\begin{equation}\label{2}
	\partial_t u(t) + \lambda_1 A_1\, u(t) + \cdots + \lambda_n A_n\,u(t) = 0, \quad t > 0,
\end{equation}
induced by the celebrated Fourier law in the framework of thermal transmission and then exported, among others, to chemistry, metallurgy, and dynamic of populations (see \cite{Ghez}).

In the case where the coefficients $\lambda_1, \dots, \lambda_n$ are known, well-posedness results for the aforementioned problem \eqref{1} can be achieved. This is the so-called \emph{direct problem}.

In particular, if $\alpha = 1$, existence and uniqueness of a solution $u(t)$ originating from an initial datum $u(0) = u_0$ in $H$ can be proven using the semigroup of operator approach which may also cover non-autonomous equations featuring source terms (see, e.g., \cite{Kato66} and \cite{Tanabe}).\\
As concerning the fractional case $\alpha \in (0,1)$, the literature on direct problems is huge; among the others, we refer to \cite{GW}, \cite{bazthesis}, \cite{Dieth}, \cite{kubyam} and the references listed there. Moreover, convergence results of the solution of \eqref{1} to the one of its singular limit \eqref{2} as $\alpha \to 1^{-}$ (in the case of a single operator) have been proved in \cite{Carvalho}.

\smallskip

 However, measurements of the quantities $\lambda_1, \dots, \lambda_n$ are not always available, and it is worthwhile mentioning that experimental approaches aimed to the determination of diffusion coefficients in different types of materials have been widely developed (see, for instance, \cite[Chapter 10]{Crank}, \cite{HdA} and the references therein).

Accordingly, the main goal of this paper is the investigation  of the \emph{inverse problem} consisting in the identification of the constants $\lambda_1, \dots, \lambda_n$, supposed to be unknown, along with $u$, under the overdeterminating conditions given  by the energy functionals
\begin{equation}
	\nonumber \langle  A_1 u(\bar T), u(\bar T) \rangle = \varphi_1, \quad \cdots \quad , \langle  A_n u(\bar T), u(\bar T) \rangle  = \varphi_n.
\end{equation}
Here $u(\bar T)$ is the solution-to-be $u(t)$ at a prescribed time-instant $t = \bar T$ and $\langle \cdot, \cdot\rangle$ is the scalar product in $H$. The relevant question to be addressed is whether the information provided by the $n$ measurements $\varphi_i$ is enough in order to uniquely reconstruct the $n$ coefficients $\lambda_i$. Thus, our main result concerns the uniqueness of the unknown $(u,\lambda_1, \dots, \lambda_n)$ in the fractional case $\alpha \in (0,1)$, provided that the initial datum $u_0$ and the additional measurements $(\varphi_1, \dots, \varphi_n)$ fulfill suitable assumptions (cf. Subsection \ref{icp}). Note that a conditioned existence result is provided, complemented with suitable numerical calculations which, to the best of our knowledge, appear to be new in the literature.

We also investigate the convergence of the solution $(u,\lambda_1, \cdots, \lambda_n)$ of the fractional equation \eqref{1} to the one of the diffusion equation \eqref{2} as $\alpha \to 1^{-}$. We remark that the well-posedness of the inverse problem in the case $\alpha = 1$ has already been studied in \cite{Mola1}. 


Applications of the abstract results mentioned above are provided. In particular, the problem of recovering diffusion coefficients in a fractional  heat equation is analyzed on bounded domains. Furthermore, we also apply our results to the identification of the Lam\'e constants of a plate in an elasticity fractional equation under the structural assumption that the inertial forces are neglected. 

\smallskip

As concerning the state of the art, the literature on inverse problems for fractional-in-time heat-type equations is not vast. As far as we know, all the results for inverse problems for fractional-in-time evolution equations deal with only one operator $A$, possibly non-autonomous, and identify the source term $F(t)$ and/or the time derivative order $\alpha$.
Very recently in \cite{FloGYam}, the authors study a fractional-in-time evolution equation in a Banach space.
After proving well-posedness results for the direct problem, they apply them to an inverse problem of determining an initial value and establishing the uniqueness.\\
Additional results concerning inverse source and backward problems associated with the Caputo fractional time derivative and the fractional Laplacian operator in space can be found in \cite{QiaoXiongHan}.
Moreover, the determination of the fractional order $\alpha$ and the diffusion coefficient from boundary data are treated in \cite{CNYY} in the one-dimensional case.
Other one-dimensional inverse problems for fractional-in-time evolution equations are considered in \cite{JingJiaSong} and \cite{Durdiev}.\\
Direct and inverse problems for the nonhomogeneous time-fractional heat equation with a time-dependent diffusion coefficient for a positive operator are considered in \cite{SerRuzTok} under linear overdeterminating conditions.\\
A numerical procedure for recovering a space-dependent diffusion coefficient for fractional evolution equations on the whole $\R^n$ using Fourier transforms is studied in \cite{Jin-Lu-Quan-Zhou}.\\
Taking into account the results mentioned above, our abstract approach allows us to consider a wide class of problems, and our inverse problem (which is intrinsically nonlinear) differs (to the best of our knowledge) from most of the ones in the literature.


\medskip

We conclude by mentioning possible next directions of our investigation. First, they include the well-posedness of the problem presented in this paper by providing continuous dependence estimates with respect to the initial data, plus possible perturbation theory for the operators $A_1, \cdots, A_n$ using Mosco convergence. This would allow us to apply our abstract results to dynamics on fractal sets, thus extending the theory recently developed in \cite{CLMR}.
	
	 \noindent A further aim could be the analysis of the non-autonomous case, focusing on the identification of time-dependent diffusion coefficients $(\lambda_1(t), \cdots, \lambda_n(t))$ along with the overdeterminating conditions of the same type as above on the whole interval $[0,\bar T]$. However, such a problem demands an extension of the techniques displayed in this paper to infinite-dimensional spaces, thus covering the results devised in the diagonal case in \cite{FMR}.

	\medskip
	
    The paper is organized as follows.\\
    Section \ref{preliminaries} contains the functional setting and the notion of Caputo-type fractional derivative.\\
    In Section \ref{abstract} we present the abstract inverse problem object of our study and we state the main results.\\
    Section \ref{main1} is devoted to the proof of  existence and uniqueness results for the solution of the inverse problem with fractional time derivative.\\
    Section \ref{Numerical simulations} contains some numerical simulations which are complementary to the existence result. A deeper discussion can be found in Appendix \ref{numerical}.\\
    In Section \ref{main2} we prove that the solution of the inverse problem \eqref{1} converges to the solution of the inverse problem \eqref{2} as $\alpha\to 1^-$.\\
    Some applications are presented in  Section \ref{applications}.\\

	

	\bigskip

	\section{Preliminaries}\label{preliminaries}
	
	\subsection{Notation and functional setting}\label{notazioni}
	
	Let $H$ be a (real) separable Hilbert space endowed with an inner product $\langle\cdot,\cdot\rangle$ and its related norm $\|\cdot\|.$ We will use the notation $\R_+=(0,\infty)$, $\R_+^{n}=(0,\infty)^n$ and $\overline{\R}_+=[0,\infty)$. 
	
	The basic algebraic facts related to Gramian matrices are well known, but for the reader's convenience, we recall the following result.
	
	\begin{proposition}\label{Gram}
		Let $y_{1},\dots,y_{n} \in H$  and consider the related $(n \times n)$-Gram matrix
		$$
		G = \left[ \langle y_{i},y_{j} \rangle \right]_{i,j = 1, \cdots,n}.
		$$
		Then $G$ is positive defined if and only if the vectors $\{y_{1},\cdots,y_{n}\}$ are linearly independent.
	\end{proposition}
	
	Let $A$ be a non-negative self-adjoint operator acting on $H$ with dense domain $D(A)$ and range $R(A)$, i.e., $A = A^{\ast}$ and $\langle x,Ax\rangle \geq 0$, for $x \in D(A)$.

	\smallskip
	
	Recall that, see, e.g., \cite{Kato66}, 
	the square root of $A$ is given by
	\begin{equation*}
		A^{1/2}x = \frac{1}{\pi}\int_{0}^{\infty}
		\xi^{-1/2}(A+\xi)^{-1}Ax\,\de\xi,
		\quad x \in D(A),
	\end{equation*}
	and displays the  property
	\begin{align*}
		\langle A x,x \rangle = \|A^{1/2}x\|^2 \leq \|x\|\|Ax\|,
		\quad x \in D(A).
	\end{align*}
	
	\smallskip

In the following, we assume that $A:D(A) \to H$ is a closed and self-adjoint (possibly unbounded) operator. We shall also require $A$ to be strictly positive, that is, $\langle Au,u\rangle \geq \Lambda \|u\|^2$ for some $\Lambda >0$ (\emph{coercivity constant}) and for every $u \in D(A)$. As is well known, under the above assumptions the operator $-A$ generates an analytic semigroup of contractions on $H$ denoted by
$\{e^{- tA}\}_{t \ge 0}$ (cf. \cite[Chapter 7]{Brezis2}). 

\smallskip

Here we subsume the main properties of the semigroup family that we shall use in the following.
\begin{proposition}\label{semigroup prop}
Let $\{e^{- tA}\}_{t \ge 0}$ be the semigroup generated by $-A$ on $H$. Then

\smallskip

\begin{itemize}
\item[(1)] for every $u_0 \in H$ the function $u(t) = e^{- tA}u_0$, for $t >0$, is the unique solution to the Cauchy problem
\begin{equation*}
	\begin{cases}
		\partial_tu(t) + Au(t) = 0, \quad t > 0,\\[2mm]
		u(0)=u_0,
	\end{cases}
\end{equation*}
with regularity
$$
u \in C(\overline{\R}_{+};H) \cap C^{1}(\R_{+};H) \cap C\left(\R_{+}; D(A)\right);
$$

\smallskip

\item[(2)] the identity 
$$
\dfrac{\partial}{\partial t}(e^{-t A}u_0) = -Ae^{- t A}u_0
$$
holds for all $t>0$ and all $u_0\in H$;

\smallskip

\item[(3)] if $\gamma > 0$ denotes the smallest eigenvalue of $A$, then the decay estimate
$$
\| e^{-t A} \|_{\mathcal{L}(H)} \leq e^{-\gamma t}
$$
holds for all $t\geq 0$, where $\| \cdot \|_{\mathcal{L}(H)}$ denotes the norm in the space $\mathcal{L}(H)$ of linear bounded functionals on $H$;

\smallskip

\item[(4)] for all $t,s \in \overline{\R}_{+}$ it holds
$$
e^{-t A}e^{-s A} = e^{-(t+s) A} \quad (\text{\emph{semigroup identity}}).
$$

\smallskip

\item[(5)] For all $t \in \overline{\R}_{+}$ the operator $e^{-t A}$ is self-adjoint and injective in $H$.
\end{itemize}
\end{proposition}

\medskip


\subsection{Time fractional derivatives}\label{fractional}

Here we recall some main facts concerning time-fractional derivatives. Following the notation introduced in \cite{GW}, for all $\alpha \in (0,1)$ we set
\begin{equation*}
	g_{\alpha}(t)
	=
	\begin{cases}
		\dfrac{t^{-1+\alpha}}{\Gamma(\alpha)}, &  t >0,\\[3mm]
		0,  & t \leq 0,
	\end{cases}
\end{equation*}
where $\Gamma$ denotes the usual Gamma function. 

\begin{definition}\label{2.1.1}
Let $Y$ be a Banach space, $T>0$ and let $f\in C([0,T];Y)$ be such that $g_{1-\alpha}\ast f \in W^{1,1} ((0, T); Y)$.
\begin{itemize}
	\item[i)] The \emph{Riemann-Liouville} fractional derivative of order $\alpha \in  (0,1)$  is defined as follows:
$$D^\alpha_t f(t):=\frac{\de}{\de t}(g_{1-\alpha}\ast f)(t)=\frac{\de}{\de t}\int_0^t g_{1-\alpha}(t-\tau) f(\tau)\,\de\tau,$$
for a.e. $t \in (0, T ]$.
	\item[ii)] The \emph{Caputo-type} fractional derivative of order $\alpha \in  (0,1)$ is defined as follows:
$$\partial^\alpha_t  f(t):= D^\alpha_t (f(t)-f (0)),$$
for a.e. $t \in (0,T]$.
\end{itemize}
\end{definition}


We stress the fact that Definition \ref{2.1.1}-$ii)$ gives a weaker definition of (Caputo) fractional derivative with respect to the original one (see \cite{CAPUTO}), since $f$ is not assumed to be differentiable. Moreover, it holds that $\partial^\alpha_t  (c) = 0$ for every constant $c\in\R$.\\
We refer to \cite{Dieth} for further details on fractional derivatives.

\medskip

Now, under the assumptions introduced in the Subsection \ref{notazioni}, we consider $u_0 \in H$ and $f: (0,T) \to H$ and we introduce the (nonhomogeneous) linear fractional Cauchy problem
\begin{equation}\label{fcp}
	\begin{cases}
		\partial_t^{\alpha} u(t) + Au(t) = f(t), & t \in (0,T),\\[3mm]
		u(0) = u_0.
	\end{cases}
\end{equation}
Next, we define a proper notion of solution for the problem introduced above (cf. \cite[Definition 2.1.4]{GW}).
\begin{definition}\label{strongsol}
	We say that $u \in C\left([0,T);H\right)$ is a strong solution for \eqref{fcp} in the interval $[0,T]$ if $u(0) = u_0$ and, for any $T_1, T_2\in (0,T)$, $T_1\le T_2$,
	\begin{itemize}
		\item[\emph{i)}] $u(t) \in D(A)$ for all $t \in [T_1,T_2]$
		\smallskip
		\item[\emph{ii)}] $\partial_t^{\alpha} u \in C\left([T_1,T_2];H\right)$;
		\smallskip
		\item[\emph{iii)}] the differential equation $\partial_t^{\alpha} u(t) + Au(t) = f(t)$ is fulfilled on $[T_1,T_2]$.
	\end{itemize}
\end{definition}

The following statement \cite[Theorem 2.1.7]{GW} concerns the well-posedness for \eqref{fcp}, and displays a useful representation formula for its solution. We first recall the  definition of the Wright type function (see \cite[Formula (28)]{GLM}):
\begin{equation}\notag
\Phi_\alpha(z):=\sum_{n=0}^\infty\frac{(-z)^n}{n!\Gamma(-\alpha n+1-\alpha)},\quad 0<\alpha<1,\,z\in \mathbb{C}.
\end{equation}

From \cite[page 14]{bazthesis}, it holds that $\Phi_\alpha(t)$ is a probability density function, i.e.
\begin{equation}\notag
\Phi_\alpha(t)\geq 0\quad\text{if }t>0,\quad\int_0^{+\infty}\Phi_\alpha(t)\,\de t=1.
\end{equation}
In addition, we have that
$$\int_0^{\infty}t^p\Phi_{\alpha}(t)\, \,\de t=\dfrac{\Gamma (p+1)}{\Gamma(\alpha p+1)}, \quad p>-1, \quad 0<\alpha<1,$$
(see, e.g., \cite{GLM}).

\noindent For more properties about the Wright function, among the others we refer to \cite{bazthesis}, \cite{GLM}, \cite{wright}.

According to \cite[Chapter 2]{GW}), we recall the following result.

\begin{theorem}\label{formula}
	Let $f: (0,T) \to H$ fulfill the following assumptions:
	\begin{itemize}
		\item[(a)] $f \in C^{0,\beta}\left((0,T);H\right)$ for some $\beta > 0$;
		\item[(b)] there exists $r  > 1/\alpha$ such that $\int_{0}^{T_0}\|f(t)\|^r\,\de t < \infty$ for some $T_0>0$.
	\end{itemize}
	
	Then, there exists a unique strong solution of \eqref{fcp} (in the sense of Definition \ref{strongsol}), given by  
	$$
	u(t) 
	= 
	S_{\alpha}(t)u_0
	+
	\int_{0}^{t}P_{\alpha}(t-s)f(s)\,\de s, \quad t \in (0,T),
	$$
	where, for $t \geq 0$ and $x\in H$, we define the (bounded) operators 
	$$
	S_{\alpha}(t) x
	=
	\displaystyle \int_0^\infty\Phi_\alpha(\tau)e^{-\tau t^\alpha A}x\,\de\tau
	\quad
	\text{and}
	\quad
	P_{\alpha}(t) x
	=
	\alpha t^{-1+\alpha}\displaystyle\int_0^\infty\tau\Phi_\alpha(\tau)e^{-\tau t^\alpha A}x\,\de\tau.
	$$
	Moreover, the following estimates hold:
	$$
	\|S_{\alpha}(t) x\| \leq \|x\|
	\quad
	\text{and}
	\quad
	\|P_{\alpha}(t) x\| \leq \alpha \dfrac{\Gamma(2)}{\Gamma(1+\alpha)}\|x\|t^{-1+\alpha},
	\quad
	t > 0,
	\quad x \in H.
	$$
\end{theorem}

\bigskip

\section{Abstract inverse problem and main results}\label{abstract}

In this Section, we state the main abstract results of this paper. The first one concerns existence and uniqueness of the solution to the inverse problem consisting of the identification of an arbitrary number of constant coefficients $\lambda_1, \cdots, \lambda_n >0$, along with a $H$-valued function $u$, in a family of evolution linear Cauchy problems in $H$ displaying a fractional derivative of order $0 <\alpha \leq 1$ ($\alpha = 1$ being the ordinary first order time derivative) under the same number of energy measurements.

From now on,  we assume that $A_i : D(A_i)\subset H\to H$, for $i=1\dots,n$, are  closed, self-adjoint and strictly positive (possibly unbounded) operators on $H$. Then, for any $n$-tuple $\boldsymbol{\lambda} = (\lambda_1, \dots, \lambda_n) \in \R_+^n$, we consider the operator
$$
A(\boldsymbol{\lambda}) = \lambda_1 A_1 + \cdots + \lambda_n A_n,
$$
 which is defined on its natural domain $D(A(\boldsymbol{\lambda})) =\cap_{i = 1}^{n} D(A_{i}):=D$, endowed with the norm $\|x\|_{D}  = \sum_{i = 1}^{n}\|A_i x\|$. Here and in the following, we assume that the operators $A_i$, $i=1,\dots,n$, are mutually commuting on $D$, that is $\langle A_i x, A_jy \rangle = \langle A_j x, A_i y \rangle$ for all $x,y \in D$ and all $i,j = 1,\dots,n$. Clearly, for any $\boldsymbol{\lambda}\in \R_+^n$, also $A(\boldsymbol{\lambda})$ has the same properties  as $A$ in Section \ref{notazioni}, and commutes with $A(\boldsymbol{\mu})$ for any choice of $\boldsymbol{\mu} \in \R_+^n$.


\smallskip

\subsection{The inverse conduction problem}\label{icp} Following the functional setting discussed in Section \ref{preliminaries} and the notation introduced above, we can formally state the abstract formulation of the family of inverse problems, depending on $0<\alpha \leq 1$, which we will investigate.  

\begin{problemPa} 
Find a function $u: \overline{\R}_{+} \to H$ and the coefficients $\boldsymbol{\lambda} \in \R_+^n$ fulfilling the Cauchy problem
\begin{equation}\label{eq}
\begin{cases}
\partial_t^{\alpha} u(t) + A(\boldsymbol{\lambda}) u(t) = 0, \quad t > 0,\\[2mm]
u(0)=u_0,
\end{cases}
\end{equation}
and the additional conditions
\begin{equation}\label{ac}
\langle  A_1 u(\bar T), u(\bar T) \rangle = \varphi_1,\;\dots\;, \langle  A_n u(\bar T), u(\bar T) \rangle = \varphi_n,
\end{equation}
where $u_0\in H$ and  $\boldsymbol{\varphi} = (\varphi_1, \dots, \varphi_n) \in \R_+^n$ are given, and $\bar T > 0$ is a fixed measuring time-instant.
\end{problemPa}

\begin{remark}
We recall that, for $\alpha=1$, the problem $\boldsymbol{P_1}$, featuring an ordinary first order time derivative, has been previously studied in \cite{Mola1}, where uniqueness and continuous dependence results have been provided. In the following, we shall extend the techniques introduced in \cite{Mola1} to the fractional case $0<\alpha < 1$.
\end{remark}

Before stating the existence and uniqueness result for the strong solution of \eqref{eq}, we make some formal computations.\\
Taking $A = A(\boldsymbol{\lambda})$, from Theorem \ref{formula} it follows that there exists a unique strong solution of \eqref{eq} which can be represented as 
$$
u_{\alpha}(t) = \int_0^\infty\Phi_\alpha(\tau)e^{-\tau t^\alpha A(\boldsymbol{\lambda})}u_0\,\de\tau, \quad t>0, 
$$
where $\left\{e^{-t A(\boldsymbol{\lambda})}\right\}_{t \geq 0}$ is the analytic semigroup generated by $-A(\boldsymbol{\lambda})$. Replacing the above formula in the left-hand-side of the overdeterminating conditions \eqref{ac} at the instant $t = \bar T$, and using the semigroup identity and the self-adjointness of the semigroups in Proposition \ref{semigroup prop}-(4) and (5), we get
\begin{align*}
\langle  A_i u_{\alpha}(\bar T), u_{\alpha}(\bar T) \rangle 
= & 
\left\langle \int_0^\infty\Phi_\alpha(\tau) A_ie^{-\tau \bar T^\alpha A(\boldsymbol{\lambda})}u_0 \,\de\tau, \int_0^\infty \Phi_\alpha(\sigma)e^{-\sigma\bar T^\alpha A(\boldsymbol{\lambda})}u_0\,\de\sigma \right\rangle
\\[3mm] 
= & 
\int_0^\infty\int_0^\infty\Phi_\alpha(\tau)\Phi_\alpha(\sigma)\left\langle  A_ie^{-\frac{\tau + \sigma}{2} \bar T^\alpha A(\boldsymbol{\lambda})}u_0, e^{-\frac{\tau + \sigma}{2} \bar T^\alpha A(\boldsymbol{\lambda})}u_0 \right\rangle\,\de\tau\,\de\sigma
\end{align*}
for all $i = 1, \dots, n$. Then, going back to \eqref{ac}, we deduce that the unknown constants $\boldsymbol{\lambda}$ must fulfill the following system of nonlinear equations:
\begin{equation}\label{system}
\boldsymbol{\F}_{\alpha}(\boldsymbol{\lambda}) 
=
\boldsymbol{\varphi},
\end{equation} 
for any fixed $\boldsymbol{\varphi} \in \mathbb{R}_{+}^{n}$,
where
the map
$$
\boldsymbol{\F}_{\alpha} = (\F_{\alpha,1},\cdots,\F_{\alpha,n}):\mathbb{R}_{+}^{n} \to \mathbb{R}_{+}^{n}
$$
is component-wise defined as
$$
\F_{\alpha,i}(\boldsymbol{\lambda}) = \int_0^\infty\int_0^\infty\Phi_\alpha(\tau)\Phi_\alpha(\sigma)\left\langle A_ie^{-\frac{\tau + \sigma}{2} \bar T^\alpha A(\boldsymbol{\lambda})}u_0, e^{-\frac{\tau + \sigma}{2} \bar T^\alpha A(\boldsymbol{\lambda})}u_0 \right\rangle\,\de\tau\,\de\sigma, \quad i=1,\dots,n.
$$

Notice that \eqref{system} only depends on the unknown vector $\boldsymbol{\lambda}$.
As it is clear, the main features of the identification problem we are investigating are related to the map $\boldsymbol{\F}_{\alpha}$. In the sequel, we shall provide suitable assumptions on $u_0$ and $\boldsymbol{\varphi}$ in order to prove that $\boldsymbol{\F}_{\alpha}$ is injective. As a by-product, this will imply that the solution to the nonlinear equation \eqref{system} is unique. 

We introduce the following definitions.

\begin{definition}
An initial datum $u_0 \in D$ is {\it admissible} for the sake of our investigation if and only if
$$
\{A_{1}u_0,\cdots, A_{n}u_0\} \text{ are linearly independent vectors in $H$.}
$$
In the sequel, we shall denote by $\A$ the subset of $D$ consisting of all admissible initial data.
\end{definition}

\begin{definition}\label{def compat}
For any fixed admissible $u_0 \in \A$, we define an additional measurement $\boldsymbol{\varphi} \in \R_+^{n}$ to be {\it compatible} if and only if
$$
\boldsymbol{\varphi} \in \mathrm{Im}\,\boldsymbol{\F}_{\alpha}. 
$$
\end{definition}

We first point out an important consequence of admissibility for an initial datum $u_0 \in D$. 
\begin{lemma}\label{transport}
Let $u_0 \in \A$. Then $e^{-t A(\boldsymbol{\lambda})}u_0 \in \A$ for all $t > 0$ and all $\boldsymbol{\lambda} \in \R_{+}^{n}$.
\end{lemma}
The proof of the above lemma is a direct consequence of the injectivity of the semigroup operator (cf. Proposition \ref{semigroup prop}-(5)) and the commutativity between the semigroup and the operators $A_i$. 

\subsection{Main results} We are now in the position to state our main results. The first one concerns existence and uniqueness of the solutions of $\boldsymbol{P_{\alpha}}$ for any $0<\alpha \leq 1$.

\begin{theorem}\label{mainresult}
Fixed $T>0$, for any admissible initial datum $u_0\in \A$ and all compatible additional measurements $\boldsymbol{\varphi}\in \mathrm{Im}\,\boldsymbol{\F}_{\alpha}$, there exists a unique solution $\left(u_{\alpha},\boldsymbol{\lambda}_{\alpha}\right)$ to Problem $\boldsymbol{P_{\alpha}}$, with
$$
u_{\alpha} \in C([T_1,T_2];D) \cap  C([0,T);H),\; \Dta u_\alpha\in C([T_1,T_2];H),\; 0<T_1\leq T_2<T,$$
and $\boldsymbol{\lambda}_{\alpha}  \in  \R_+^n
$, where $\boldsymbol{\varphi}$ satisfies \eqref{ac} with $\bar{T}\in[T_1,T_2]$,  
\end{theorem}

The second result concerns the convergence of the solutions of problems $\boldsymbol{P_{\alpha}}$ to the one of $\boldsymbol{P_{1}}$ as $\alpha \to 1^{-}$. 
\begin{theorem}\label{mainresult1}
For $u_0\in \A$, let $\left(u_{\alpha},\boldsymbol{\lambda}_{\alpha}\right)$ be the corresponding solution to problem $\boldsymbol{P_{\alpha}}$ as in Theorem \ref{mainresult}. Then
$$
u_{\alpha}(t) \to u_{1}(t) \quad\text{ in } H
$$ 
for all $t \in (0,T)$, and 
$$
\boldsymbol{\lambda}_{\alpha} \to \boldsymbol{\lambda}_{1} \quad \text{in }\R^{n}_{+}
$$
as $\alpha \to 1^{-}$. 
\end{theorem}

The proofs of Theorems \ref{mainresult} and \ref{mainresult1} will be presented in Sections \ref{main1} and \ref{main2}, respectively.

\medskip

\section{Existence and uniqueness: proof of Theorem \ref{mainresult}}\label{main1}
We divide the proof into two parts: \emph{uniqueness} and \emph{existence}. We begin from the former one, which is a crucial point in our investigation.

\subsection{Injectivity of $\boldsymbol{\F_{\alpha}}$: uniqueness}\label{injectivity}
We first point out some analytical properties of the map $\boldsymbol{\F_{\alpha}}$.

\begin{proposition}\label{properties}
Let $\boldsymbol{\F}_{\alpha}$ be as in
\eqref{system}. Then
\begin{itemize}
\item[\textbf{\emph{(a)}}] $\boldsymbol{\F}_{\alpha} \in C^{1}(\mathbb{R}_{+}^{n};\,\mathbb{R}_{+}^{n})$,
and its differential $\boldsymbol{\F}_{\alpha}'(\boldsymbol{\lambda})$ is given by the $n \times n$ Jacobi matrix
$$
-2\bar{T}^\alpha 
\left(	\int_0^\infty\int_0^\infty(\tau + \sigma)\Phi_\alpha(\tau)\Phi_\alpha(\sigma)\left\langle  A_ie^{-\frac{\tau + \sigma}{2} \bar T^\alpha A(\boldsymbol{\lambda})}u_0, A_je^{-\frac{\tau + \sigma}{2} \bar T^\alpha A(\boldsymbol{\lambda})}u_0  \right\rangle\,\de\tau\,\de\sigma\right)_{i,j=1,\cdots,n}
$$
for all $\boldsymbol{\lambda} \in \R_{+}^{n}$;
\item[\textbf{\emph{(b)}}] the Jacobi matrix $\boldsymbol{\F}_{\alpha}'(\boldsymbol{\lambda})$ is negative defined.
\end{itemize}
\end{proposition}	

\begin{proof}
Point \textbf{\emph{(a)}} is a straightforward application of the derivation formula in Proposition \ref{semigroup prop}-(2) and the commutativity of operators $A_i$.

In order to prove \textbf{\emph{(b)}}, we first notice that by the positivity of the Wright function $\Phi_\alpha$, this is equivalent to prove that the matrix
$$
\boldsymbol{\mathcal{M}}_\alpha(\boldsymbol{\lambda})
=
\left(
\left\langle  A_ie^{-\frac{\tau + \sigma}{2} \bar T^\alpha A(\boldsymbol{\lambda})}u_0, A_je^{-\frac{\tau + \sigma}{2} \bar T^\alpha A(\boldsymbol{\lambda})}u_0\right\rangle 
\right)_{i,j=1,\cdots,n}
$$
is positive defined. This is implied by the fact that $	\boldsymbol{\mathcal{M}}_\alpha(\boldsymbol{\lambda})$ is the Gram matrix in $H$ associated with the vectors
$$
A_1e^{-\frac{\tau + \sigma}{2} \bar T^\alpha A(\boldsymbol{\lambda})}u_0,
\;\dots \;, A_ne^{-\frac{\tau + \sigma}{2} \bar T^\alpha A(\boldsymbol{\lambda})}u_0
$$
which are linearly independent by Lemma \ref{transport}.
\end{proof}

Finally, the injectivity of $\boldsymbol{\F}_{\alpha}$ is a consequence of the following general result (cf. \cite[Theorem 6]{GN}).

\begin{theorem}\label{injectivity1}
Let $\mathcal{C}\subseteq \R^{n}$ be a convex nonempty set, and let
$\boldsymbol{\F} \in C^{1}(\mathcal{C};\R^{n})$ be such that its differential $\boldsymbol{\F}'(\mathbf{x})$ is positive (or negative) defined at every point $\mathbf{x} \in \mathcal{C}$. Then $\boldsymbol{\F}$ is injective on $\mathcal{C}$.
\end{theorem}

\begin{remark}
We stress that the commutativity assumption of operators $A_1,\dots,A_n$ is necessary for the sake of the uniqueness result above. In fact, in the Appendix section of \cite{MOY}, a counterexample is displayed in the case where $\alpha = 1$ and $n = 2$.
\end{remark}

\medskip


\subsection{Basic properties of $\mathrm{Im}\,\boldsymbol{\F_{\alpha}}$: existence}\label{prop Im F}

As previously mentioned, the existence of a solution to the nonlinear equation \eqref{system} is, in fact, equivalent to the compatibility condition displayed in Definition \ref{def compat}. In the following, we provide a description of the set $\mathrm{Im}\,\boldsymbol{\F_{\alpha}}$ based on analytical considerations; meanwhile, in Section \ref{Numerical simulations} a numerical description will be outlined. 

\subsubsection*{Topological properties} First, as a further consequence of the injectivity of map $\boldsymbol{\F_{\alpha}}$ proven in Subsection \ref{injectivity} and since $\boldsymbol{\F}_{\alpha}'(\boldsymbol{\lambda})$ is non singular at every $\boldsymbol{\lambda}$, we learn that its inverse map
$$
\boldsymbol{\F}_{\alpha}^{-1}: \mathrm{Im}\,\boldsymbol{\F}_{\alpha} \to \R^n_+
$$
is differentiable on its domain. Therefore, $\boldsymbol{\F}_{\alpha}$ is open. 

Moreover, by the semigroup decay estimate and the properties of function $\Phi_\alpha$, it is immediate to see that
$$
|\F_{\alpha,i}(\boldsymbol{\lambda})|
\leq
\left(\int_0^\infty\int_0^\infty\Phi_\alpha(\tau)\Phi_\alpha(\sigma)\,\de\tau\,\de\sigma\right)\|A_iu_0\|\|u_0\|
=
\|A_iu_0\|\|u_0\|, \quad i=1,\dots,n,
$$
so that we deduce the limitation
$$
|\boldsymbol{\F}_{\alpha}(\boldsymbol{\lambda})|_n \leq \|u_0\|_{D}\|u_0\|,
$$
where we denoted by $|\cdot|_{n}$ the usual Euclidean norm in $\R^{n}$.
\smallskip

In summary, we have proved the following.
\begin{proposition}
	Let $u_0\in\mathcal{A}$ be an admissible initial datum. Then $\mathrm{Im}\,\boldsymbol{\F}_{\alpha}$ is an open and path-connected subset of $\R^n_+$, bounded by $\|u_0\|_{D}\|u_0\|$ (uniformly with respect to $\alpha$ and $\boldsymbol{\lambda}$).
\end{proposition}

\section{Numerical simulations}\label{Numerical simulations} In order to fix ideas about the possible outcomes for the shape of $\mathrm{Im}\,\boldsymbol{\F}_{\alpha}$, we use numerical analysis tools to solve a discretized version of several examples of the direct problem described in Subsection \ref{bounded} (thermal conductivity on bounded sets). This allows for visualizing both the direct problem solution and, more importantly, a sketch of the map between the diffusion parameters and the output, i.e., the $L^2$ norms of the solution first derivatives. All the results presented and discussed in this Section will refer to two-dimensional problems ($n = 2$). Despite the computational tool developed for the solution of the direct problem is in principle able to also tackle the solution of three-dimensional problems, the authors believe that a 2D setup results in more informative plots of the input-to-output maps.



Here, the direct problem $\boldsymbol{P}_\alpha$ has been discretized in space using the finite element method \cite{qua-val}. In particular, the Galerkin method is implemented here making use of the software tools included in the deal.II library \cite{2024:africa.arndt.ea:deal}, and Lagrange shape functions of order one have been selected. An L1 scheme method \cite{old-spa} is used for the time integration of the unsteady problem at hand. This numerical setup allowed us to evaluate the solution of the direct problem making use of different combinations of domain $\Omega$ shapes, initial data, fractional derivative exponent $\alpha$ and diffusion coefficients $(\lambda_1,\lambda_2)$. Thus, a simulation campaign has been designed to assess the possible dependence of domain shape, initial conditions, and fractional derivative exponents on the map between the input diffusion parameters and the output solution derivative norms.

As illustrated in Fig.~\ref{fig::initial_sol}, two different shapes have been considered for the domain $\Omega$. In addition, for each domain shape, we considered two different initial conditions.  The top plots in the figure depict an L-shaped domain geometric configuration, namely
$$
\Omega_1 = \left([-1,1]\times[-1,1]\right) \setminus \left([0,1]\times[0,1]\right).
$$

The alternative geometry used is represented by the disc with eccentric round hole depicted in the bottom plots of Fig.~\ref{fig::initial_sol}. Defining as $C_{ext}$ the circle with center in $c_{ext}=\{0,0\}$ and radius $\rho_{ext}=1$, and with $C_{int}$ the circle with center in $c_{int}=\{0.2,0.2\}$ and radius $\rho_{ext}=0.3$, the effective geometry is defined as  
$$
\Omega_2 = C_{ext} \setminus C_{int}.
$$

Concerning the initial conditions, all the different combinations considered have been obtained through linear combinations of the $C^\infty$ cutoff function given by
$$
f_{cutoff}(\boldsymbol{x};\boldsymbol{x_0},s) = e^{1-\frac{1}{1-||\boldsymbol{x}-\boldsymbol{x_0}||^2/s^2}} \quad \boldsymbol{x}\in \Omega\subset \mathbb{R}^2.
$$
Here $\Omega=\Omega_1$, or $\Omega=\Omega_2$, $s\in\mathbb{R}^+$ is the cutoff radius, and $\boldsymbol{x_0}\in\mathbb{R}^2$ is the center of the cutoff function. 

The two different initial conditions used for $\Omega_1$ and depicted in the first row of plots in Fig.~\ref{fig::initial_sol} read
\begin{eqnarray*}
\phi_1^{\Omega_1}(\boldsymbol{x}) &=& f_{cutoff}(\boldsymbol{x},\{-0.5,-0.5\},0.3) \\
\phi_2^{\Omega_1}(\boldsymbol{x}) &=& f_{cutoff}(\boldsymbol{x},\{-0.5,-0.3\},0.35)+f_{cutoff}(\boldsymbol{x},\{-0.55,0.55\},0.25)
\end{eqnarray*}

The initial conditions used with the $\Omega_2$ test cases read instead
\begin{eqnarray*}
\phi_1^{\Omega_2}(\boldsymbol{x}) &=& f_{cutoff}(\boldsymbol{x},\{-0.5,-0.5\},0.35) \\
\phi_2^{\Omega_2}(\boldsymbol{x}) &=& f_{cutoff}(\boldsymbol{x},\{-0.1,-0.55\},0.35)+f_{cutoff}(\boldsymbol{x},\{-0.55,0.55\},0.25)
\end{eqnarray*}
and are portrayed in the second row of Fig. \ref{fig::initial_sol} (left and right plots, respectively).  In the numerical algorithm used, the latter condition is enforced by means of a suitable set of constraints introduced in the linear algebraic system resulting from the discretization and solved at each time step. We point out that, as can be appreciated in all the plots, all the initial data functions $\phi_1^{\Omega_1},\phi_2^{\Omega_1},\phi_1^{\Omega_2},\phi_2^{\Omega_2}$ present a null boundary value, in order  to be compatible with the homogeneous Dirichlet boundary condition imposed.  Finally, we remark that all the plots presented in the Fig.  \ref{fig::initial_sol} also include the computational grid used for the simulations.

\begin{figure}[!h]
\begin{tabular}{cc}
\includegraphics[width=0.48\textwidth]{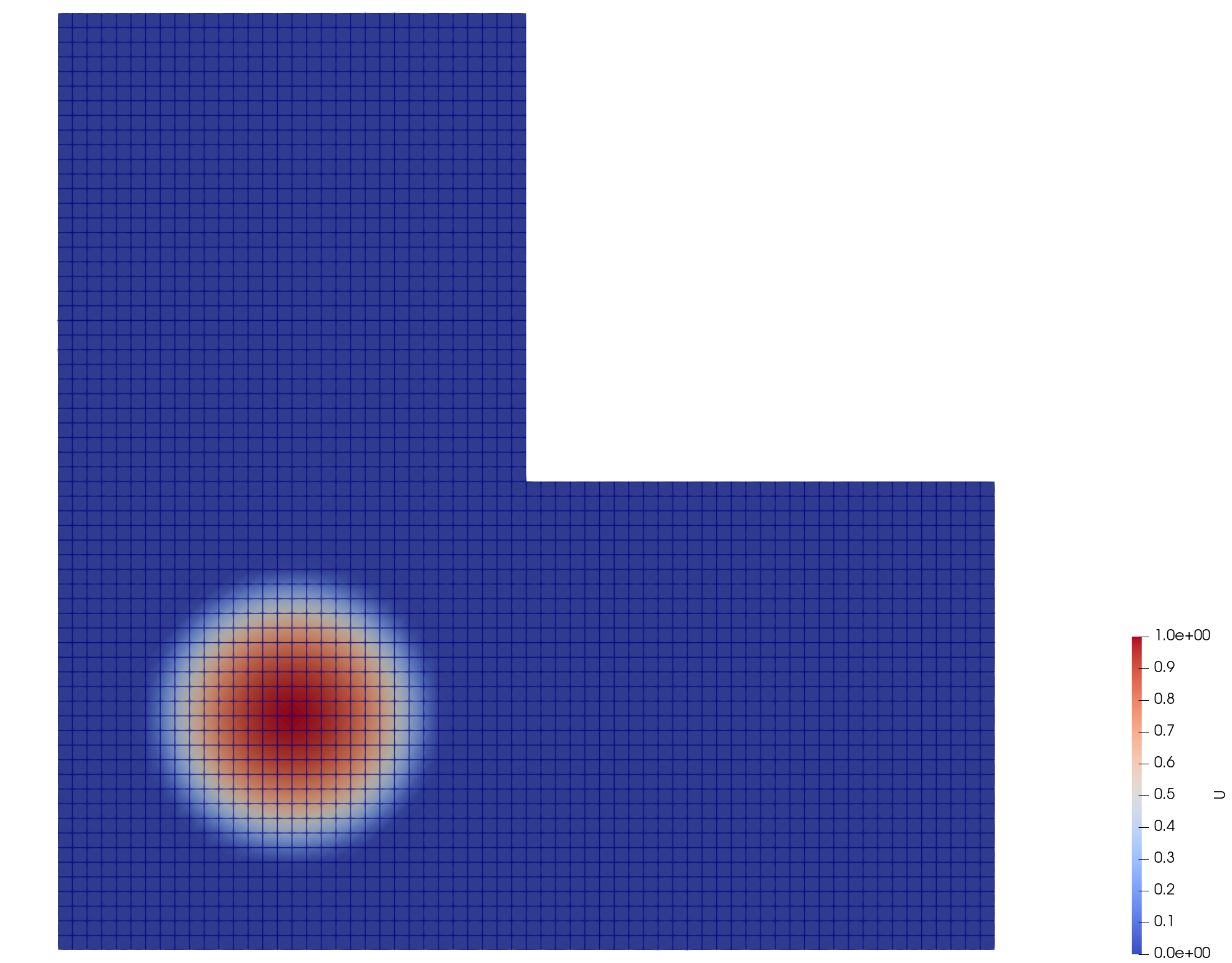}     &
\includegraphics[width=0.48\textwidth]{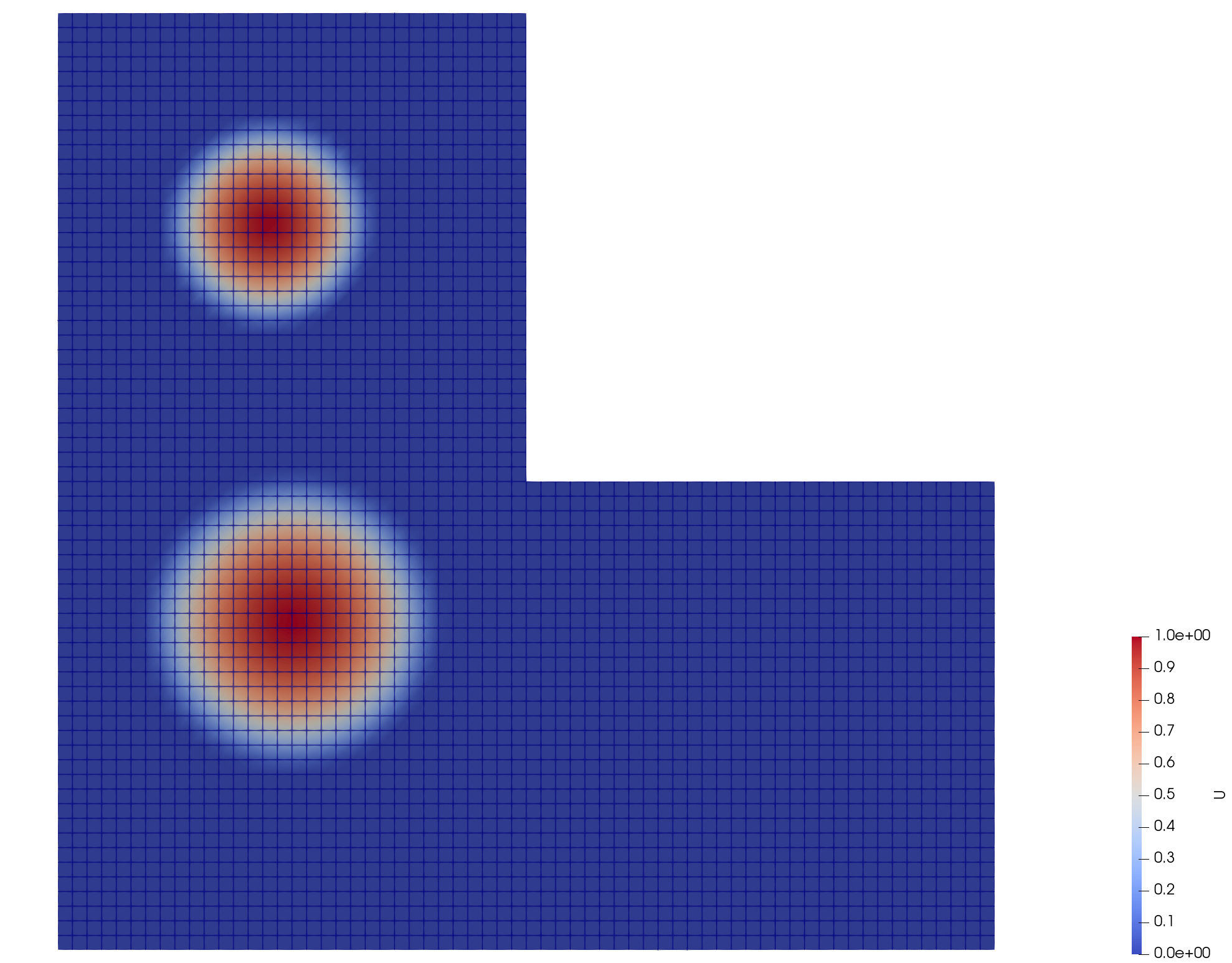} \\
a) & b) \\
\includegraphics[width=0.48\textwidth]{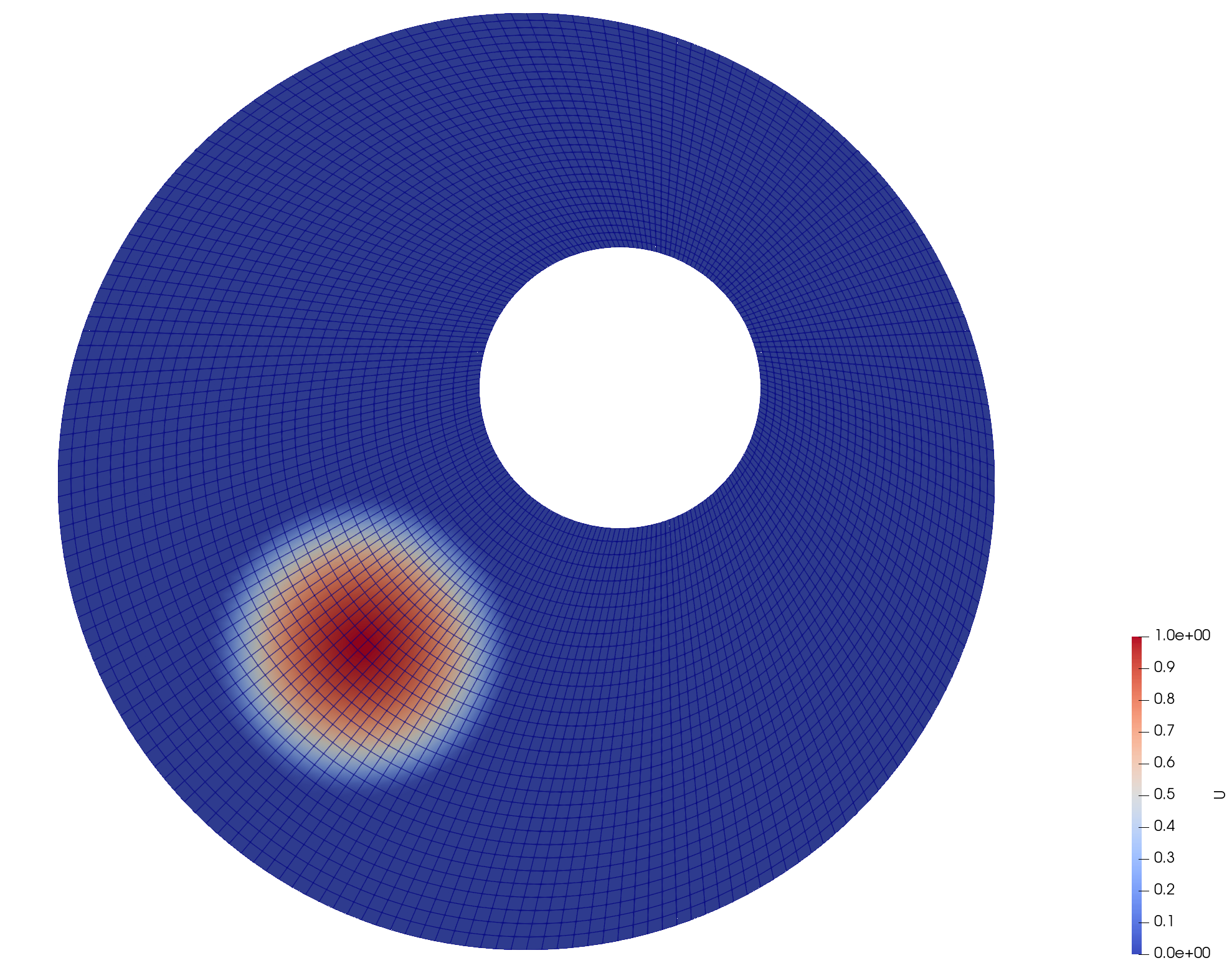}  &
\includegraphics[width=0.48\textwidth]{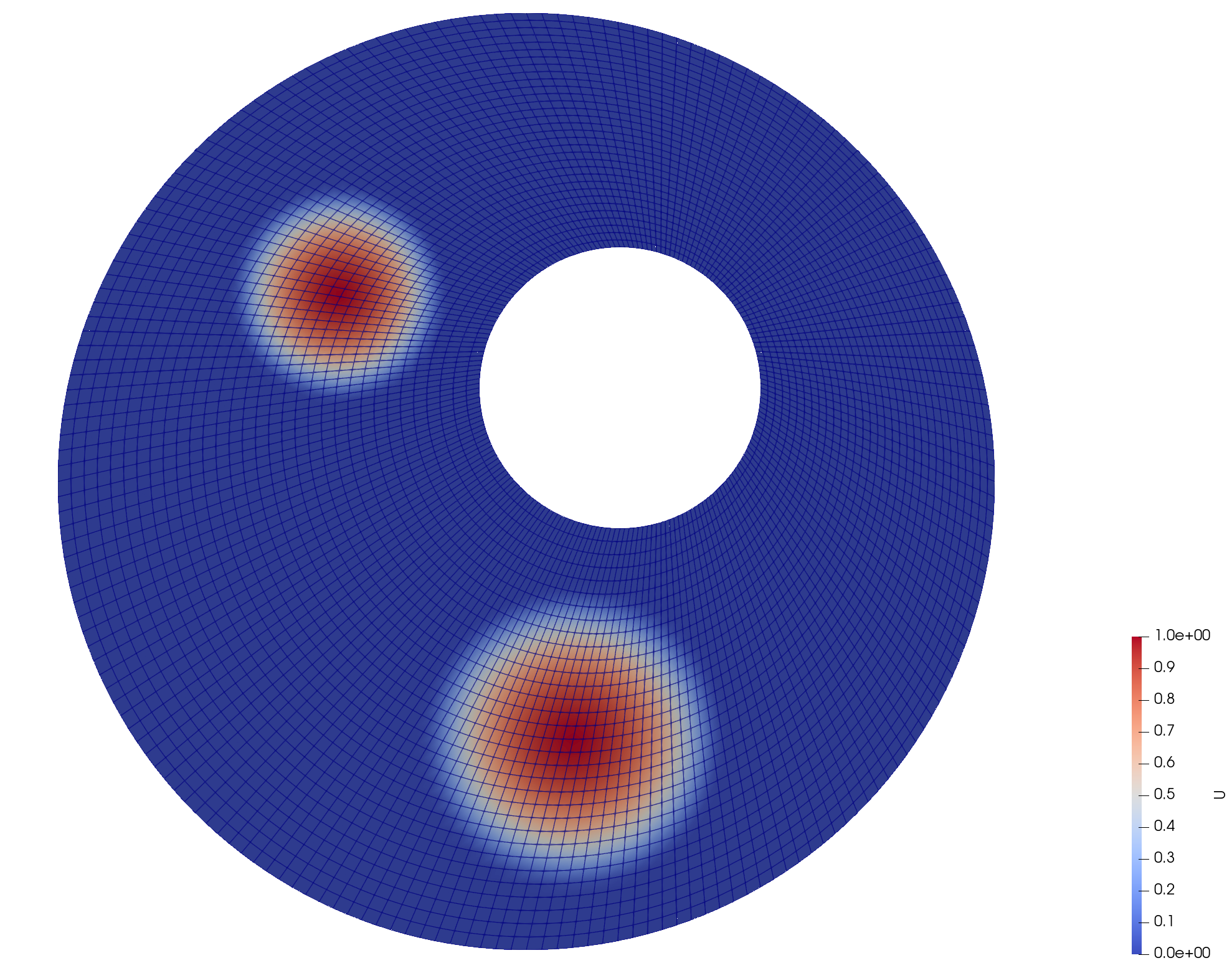} \\
c) & d) 
\end{tabular}
	
	\caption{A sketch of the computational domains used for the forward problem simulations. The top row plots depict the L-shaped domain, while the bottom plots are referred to the circle with eccentric hole. Each plot also features the contour plot of the initial condition function prescribed: a) symmetrically centered $C^\infty$ cutoff function; b) two asymmetrically centered $C^\infty$ cutoff functions of different radii; c) symmetrically centered $C^\infty$ cutoff function; d) two asymmetrically centered $C^\infty$ cutoff functions of different radii. All the plots also show the computational grid used for the corresponding forward simulation.\label{fig::initial_sol}}
\end{figure}

Each of the four combinations of domain and initial conditions described has been used to run simulations with three values of the fractional derivative exponent $\alpha = 0.25,0,75,1.0$. In addition, for each of the 12 resulting problems, the diffusion coefficient samples considered $(\lambda_1,\lambda_2)$ are the ones summarized in the $10 \times 10$ grid featured in Fig~\ref{fig::diff_params}.

\begin{figure}[!h]
	\includegraphics[width=11cm]{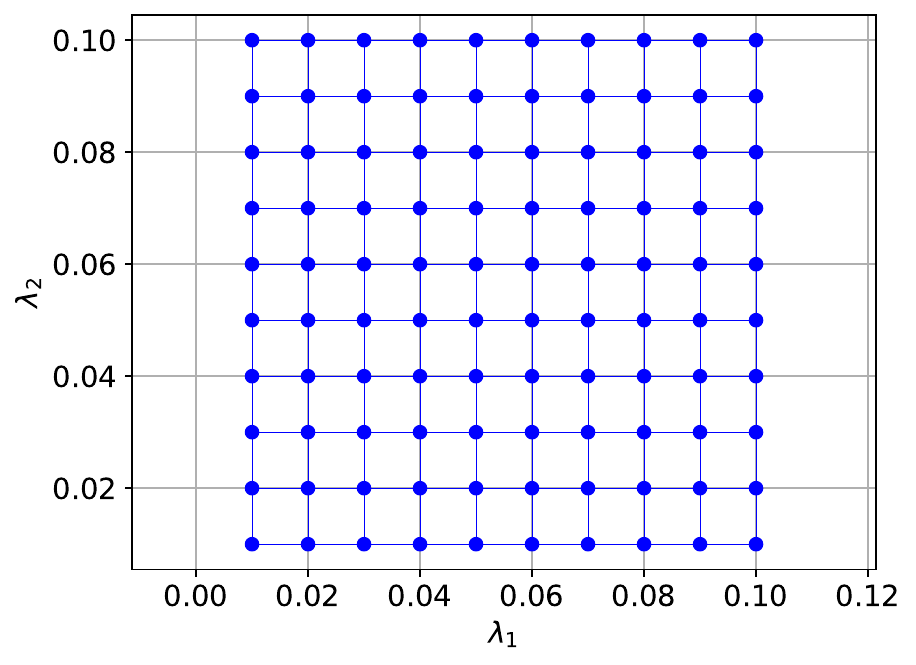}
	\caption{Diffusion parameters. \label{fig::diff_params}}
\end{figure}



All of the 1200 initial-boundary value problems considered have been solved in the time interval $t\in[0,0.04]$, and at the end of such interval, the $L^2$ norms of the solution derivatives $||u_x||_2$ and $||u_y||_2$ have been computed within the framework of a post processing step.  In the diagrams presented in Figs. \ref{fig::output_L_symm}, \ref{fig::output_L_asymm}, \ref{fig::output_eccentric_symm} and \ref{fig::output_eccentric_asymm} (see Appendix \ref{numerical}), for each configuration of $\Omega_i$ and $\phi_j^{\Omega_i}$, $i,j=1,2$ considered in the simulation campaign, the horizontal axis represents the values of the solution $x$ derivative norm $||u_x||_2$ at time $t=0.04$, while the vertical axis represents the values of the solution $y$ derivative norm $||u_y||_2$ at the same -- final -- time instant of the simulation. Thus, each point $(||u_x||_2,||u_y||_2)$ in the above plots represents the image of the input parameter pair $(\lambda_1,\lambda_2)$ through $\boldsymbol{\F}_{\alpha}$. For such a reason, the output parameters generated by the set of input parameter points can provide valuable information on the map $\boldsymbol{\F}_{\alpha}$. We also point out that the points in Fig. \ref{fig::all_results} and in Figs. \ref{fig::output_L_symm}, \ref{fig::output_L_asymm}, \ref{fig::output_eccentric_symm} and \ref{fig::output_eccentric_asymm} of Appendix \ref{numerical} are colored according to the value of fractional derivative exponent considered, namely
$$
\color{red}{\bullet} \,\,\, \color{black}{\alpha = 0.25}
\qquad
\color{verde}{\bullet} \,\,\, \color{black}{\alpha = 0.75}
\qquad
\color{black}{\bullet} \,\,\, \alpha = 1.
$$

We point out that all the four output plots presented in Fig.~\ref{fig::all_results} include, for each value of $\alpha$, a grid that has been generated with the same connectivity used for the input parameters and is illustrated in Fig. \ref{fig::diff_params}. Indeed, a first interesting observation that can be drawn from all the output plots presented is related to the way the input parameter points grid is deformed by map $\boldsymbol{\F}_{\alpha}$. Remarkably, the fact that the structure of the grid in the parameter space is preserved in the output space, and the topology of the cells is not altered by the input-to-output map, is an indirect confirmation of the fact that the solution of the inverse problem at hand is unique. On the other hand, the original input parameter space square (see Fig. \ref{fig::diff_params}) in the output is mapped by $\boldsymbol{\F}_{\alpha}$ in an image shape which -- despite remaining that of a stretched and deformed quadrilateral -- is rather difficult to identify and associate with possible characteristics of the problem. Thus, the present numerical investigation is less able to provide indications regarding the existence of the inverse problem solution. However, it must be pointed out that the topology and general shape of the output parameter plots observed are not affected by the different domain shape, initial condition, and fractional derivative exponent considered. This could in principle suggest the possibility to identify a common structure of the map image $\boldsymbol{\F}_{\alpha}$, and possibly provide ideas toward proving the existence of the inverse problem solution.

\begin{figure}[!h]
\begin{tabular}{cc}
\includegraphics[width=0.48\textwidth]{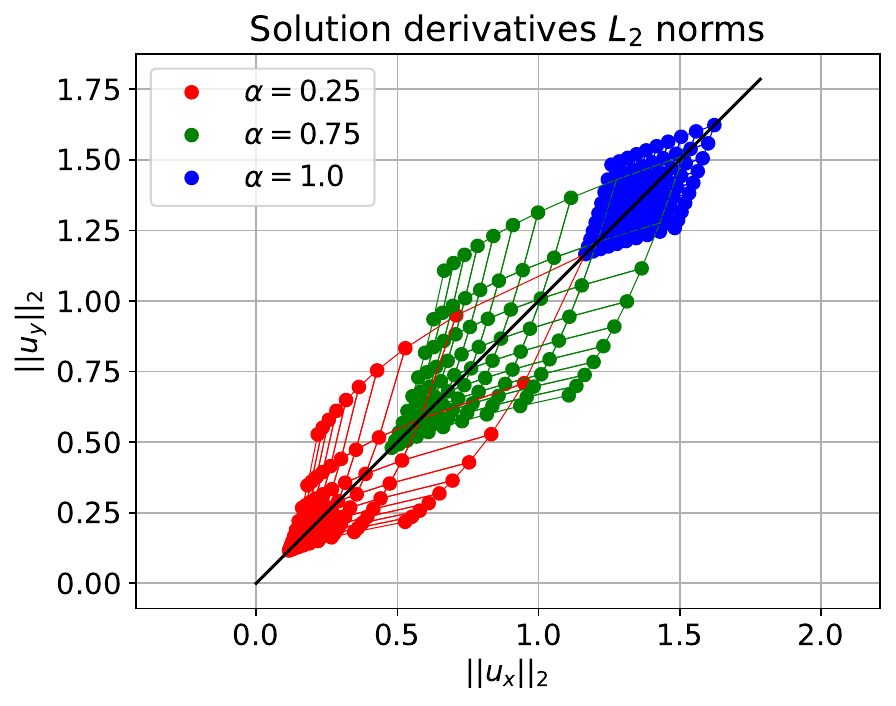}     &
\includegraphics[width=0.48\textwidth]{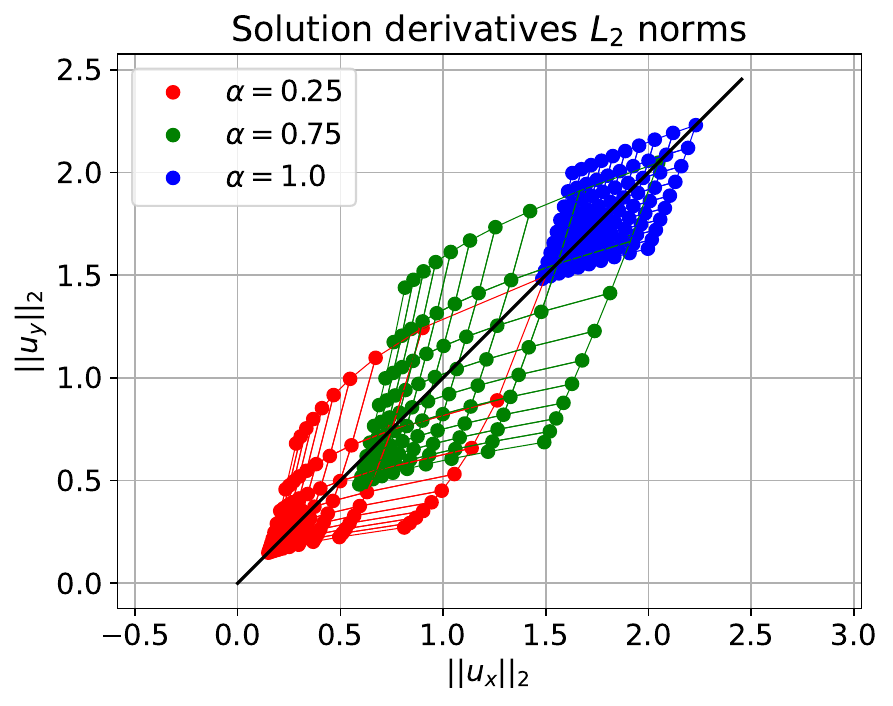} \\
a) & b) \\
\includegraphics[width=0.48\textwidth]{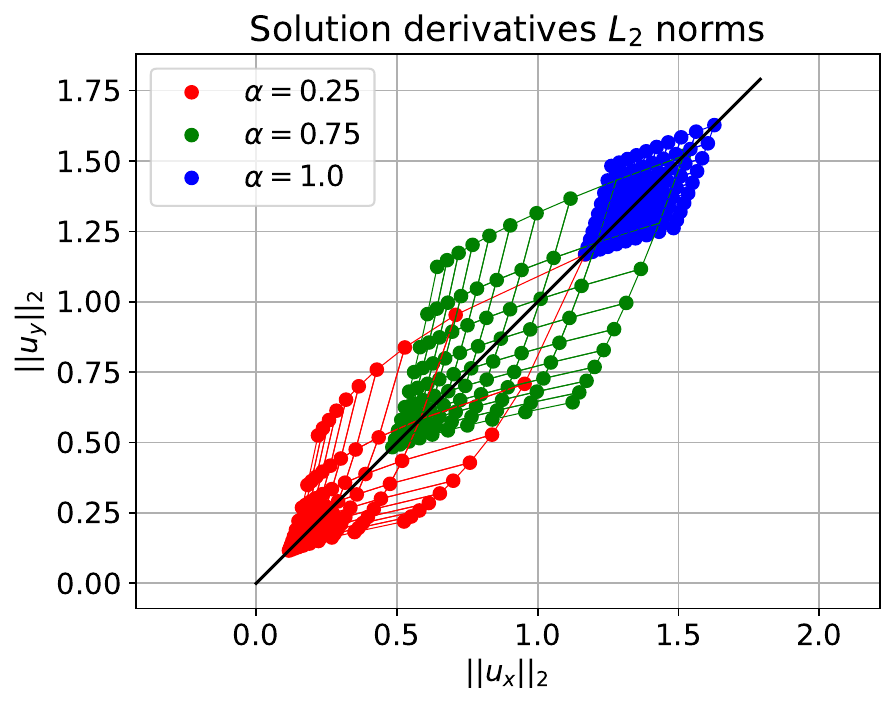}  &
\includegraphics[width=0.48\textwidth]{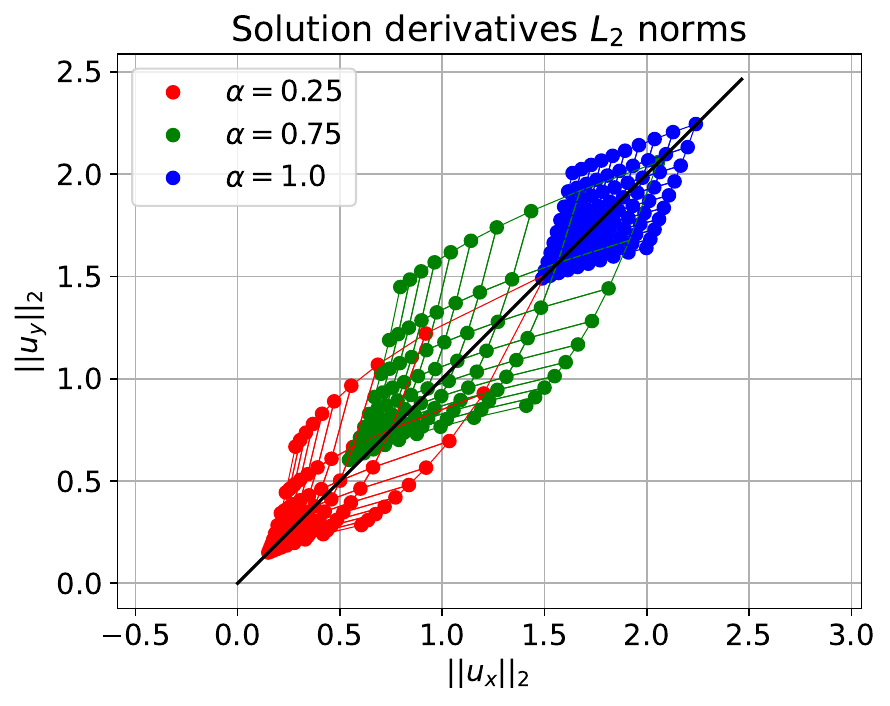} \\
c) & d) 
\end{tabular}
	
	\caption{Solution derivatives norms $||u_x||_2$ vs $||u_y||_2$ obtained using the diffusion parameters in Fig~\ref{fig::diff_params} in the problem with: a) L-shaped domain and initial condition featuring a single symmetric $C^\infty$ cutoff function; b) L-shaped domain and initial condition featuring two asymmetrically centered $C^\infty$ cutoff functions; c) circular domain with eccentric hole and initial condition featuring a single symmetric $C^\infty$ cutoff function; d) circular domain with eccentric hole and initial condition featuring two asymmetrically centered $C^\infty$ cutoff functions.\label{fig::all_results}}
\end{figure}

A further common trait in all the four output diagrams is the fact that, as expected, the least dissipative result corresponds to the fractional derivative exponent $\alpha=1$. In fact, the higher output parameters values obtained with $\alpha=1$ indicate that the solution has been subjected by less diffusion, with respect to the ones generated using $\alpha=0.75$ and $\alpha=0.25$. Conversely, the most diffusive simulations, which also lead to generally lower values of the solution derivatives norms, are the ones associated to the lowest value of fractional derivative exponent $\alpha=0.25$ considered. Finally, we point out that as expected, plots a) and c) in Fig~\ref{fig::all_results} appear symmetric with respect to the axis bisector line. In fact, they refer to symmetric domain and boundary conditions combinations. 


\section{Convergence of solutions as $\alpha \to 1^{-}$: proof of Theorem \ref{mainresult1}}\label{main2}

In this Section, we prove the convergence of the solution of inverse problem $\boldsymbol{P_{\alpha}}$ as it is approaching its singular limit $\boldsymbol{P_{1}}$. In the following, we shall denote by $S_{\alpha, \boldsymbol{\lambda}}(t)$ and $P_{\alpha, \boldsymbol{\lambda}}(t)$ the operators introduced in Theorem \ref{formula} associated to the generator $A(\boldsymbol{\lambda})$.

\smallskip

We divide the proof into four steps.

\smallskip

\subsubsection*{Step 1} We begin by noticing that (cf. \cite[Lemma 3.12]{Carvalho}) for every fixed $u_0 \in H$ and $\boldsymbol{\lambda}$, there holds
$$
S_{\alpha, \boldsymbol{\lambda}}(t)u_0 \to S_{1, \boldsymbol{\lambda}}(t)u_0 \quad \text{as} \quad \alpha \to 1^{-},
$$
uniformly on every bounded set in $(0,\infty)$. This implies the continuity of the solution of \eqref{eq} as $\alpha \to 1^{-}$, since the continuity in the interval $(0,1)$ is a direct consequence of formulae in Theorem \ref{formula}. 

\smallskip

\subsubsection*{Step 2} Collecting the results provided in \emph{Step 1} and Proposition \ref{properties}-\textbf{\emph{(a)}}, we deduce that the $\R_+^{n}$-valued function
$$
(\alpha, \boldsymbol{\lambda}) \mapsto \boldsymbol{\F}_{\alpha}(\boldsymbol{\lambda})
$$
is continuous in $(0,1] \times \R_+^{n}$ and is $\boldsymbol{\lambda}$-differentiable for all $\alpha \in (0,1]$. 
Since, under the additional assumption $u_0 \in \A$ , the Jacobi matrix $\boldsymbol{\F}'_{\alpha}(\boldsymbol{\lambda})$ is not singular for all $\alpha$ and $\boldsymbol{\lambda}$ (cf. Proposition \ref{properties}-\textbf{\emph{(b)}}), then the equation
$$
\boldsymbol{\F}_{\alpha}(\boldsymbol{\lambda}) = \boldsymbol{\varphi} \in \mathrm{Im}\, \boldsymbol{\F}_{1}
$$
defines a continuous implicit function 
$
\boldsymbol{\lambda} = \boldsymbol{\lambda}_{\alpha}
$
in a left neighborhood of $\alpha = 1$. This proves that 
$$
\boldsymbol{\lambda}_{\alpha} \to \boldsymbol{\lambda}_{1} \quad \text{as} \quad \alpha \to 1^{-}.
$$
As a byproduct, this shows that $\boldsymbol{\varphi} \in \mathrm{Im}\, \boldsymbol{\F}_{\alpha}$ for every $\alpha$ close enough to $1$.

\smallskip

\subsubsection*{Step 3} Now, let $\alpha \in (0,1)$ and $u_0 \in \A$ be fixed. Consider two $n$-tuples of coefficients $\boldsymbol{\lambda}, \boldsymbol{\mu} \in \R_{+}^{n}$, and denote by $u$ and $v$ the respective solutions of the direct problem \eqref{eq} governed by operators $A(\boldsymbol{\lambda})$ and $A(\boldsymbol{\mu})$ and originating from the same initial datum $u_0$. Then, their difference fulfills the Cauchy problem
\begin{equation}\label{eqdiff}
\begin{cases}
\partial_t^{\alpha} (u-v)(t) + A(\boldsymbol{\lambda}) (u-v)(t) =  -\left(A(\boldsymbol{\lambda}) -A(\boldsymbol{\mu})\right) v(t), \quad t\in(0,T),\\[2mm]
(u-v)(0)=0.
\end{cases}
\end{equation}
Then, by Theorem \ref{formula} and using the commutativity of $A(\boldsymbol{\lambda})$ and $A(\boldsymbol{\mu})$, we learn 
\begin{align*}
(u-v)(t) 
& =
-\int_{0}^{t}P_{\alpha, \boldsymbol{\lambda}}(t-s)
\left(A(\boldsymbol{\lambda})- A(\boldsymbol{\mu})\right) S_{\alpha, \boldsymbol{\mu}}(s)u_0\,\de s
\\[2mm]
& =
-\int_{0}^{t}P_{\alpha, \boldsymbol{\lambda}}(t-s)
S_{\alpha, \boldsymbol{\mu}}(s)\left(A(\boldsymbol{\lambda})- A(\boldsymbol{\mu})\right)u_0\,\de s
\end{align*}
so that
\begin{align*}
\|(u-v)(t)\| 
& \leq
\alpha \dfrac{\Gamma(2)}{\Gamma(1+\alpha)}\int_{0}^{t}(t-s)^{-1+\alpha}
\left\| S_{\alpha, \boldsymbol{\mu}}(s)\left(A(\boldsymbol{\lambda})- A(\boldsymbol{\mu})\right)u_0 \right\|\,\de s
\\[2mm]
& \leq
\dfrac{\Gamma(2)}{\Gamma(1+\alpha)}t^{\alpha}
\left\|\left(A(\boldsymbol{\lambda})- A(\boldsymbol{\mu})\right)u_0\right\| 
\leq
\dfrac{\Gamma(2)}{\Gamma(1+\alpha)}t^{\alpha}\sum_{i=1}^{n}|\lambda_{i}-\mu_{i}|\|A_i u_0\|
\\[2mm]
& \leq \dfrac{\Gamma(2)}{\Gamma(1+\alpha)}t^{\alpha}\|u_0\|_{D} \max_{i=1, \cdots, n}|\lambda_{i}-\mu_{i}|.
\end{align*}
This proves the Lipschitz estimate
$$
\left\|S_{\alpha, \boldsymbol{\lambda}}(t)u_0 - S_{\alpha, \boldsymbol{\mu}}(t)u_0 \right\| \leq C_n(\alpha, t,u_0) |\boldsymbol{\lambda} - \boldsymbol{\mu}|_{n}.
$$
Here, the constant $C_n(\alpha, t,u_0)$ depends continuously on $\alpha \in (0,1)$, $t \in (0,T)$ and $u_0\in D$. The same result holds for the limiting case $\alpha = 1$ as a consequence of basic parabolic estimates.

\smallskip

\subsubsection*{Step 4} Finally, we consider the functions $u_\alpha(t)$ and $u_1(t)$ as defined in Theorem \ref{mainresult1}. Then, using the estimate provided in \emph{Step 3}, we have
\begin{align*}
\|u_\alpha(t) - u_1(t)\| 
& = 
\left\|S_{\alpha, \boldsymbol{\lambda}_\alpha}(t)u_0 - S_{1, \boldsymbol{\lambda}_1}(t)u_0 \right\|
\\[2mm]
& \leq
\left\|S_{\alpha, \boldsymbol{\lambda}_\alpha}(t)u_0 - S_{\alpha, \boldsymbol{\lambda}_1}(t)u_0 \right\|
+
\left\|S_{\alpha, \boldsymbol{\lambda}_1}(t)u_0 - S_{1, \boldsymbol{\lambda}_1}(t)u_0 \right\|
\\[2mm]
& \leq
C_n(\alpha, t,u_0) |\boldsymbol{\lambda}_{\alpha} - \boldsymbol{\lambda}_1|_{n}
+
\left\|S_{\alpha, \boldsymbol{\lambda}_1}(t)u_0 - S_{1, \boldsymbol{\lambda}_1}(t)u_0 \right\| 
\\[2mm]
& \to 0 \quad \text{as } \alpha \to 1^{-},
\end{align*}
as a consequence of the convergences stated in \emph{Step 1} and \emph{Step 2}.

\smallskip

\noindent This concludes the proof of Theorem \ref{mainresult1}.

\section{Applications}\label{applications}


%
%
%
%
%
%
%
%
%
%
%

In this last Section, we shall display concrete applications of our abstract result to a family of initial boundary value problems. In particular, in Subsection \ref{bounded}
we deal with an inverse diffusion problem for a parabolic equation on a bounded set. In Subsection \ref{lame} we study the problem of the identification of Lam\'e coefficients in a plane elasticity problem (cf. \cite{ANS} and \cite{NU} for a similar problem in the steady case).

\subsection{Realization on bounded sets}\label{bounded}

Let $\Omega_{i} \subseteq \R^{d_{i}}$, $i=1,\cdots,n$, be
bounded domains with regular boundary
$\partial\Omega_{i}$, and set $\Omega = \Omega_{1} \times \cdots\times \Omega_{n}$. Denote by $\Delta_{\mathbf{x}_{i}}$ the
realizations of the Dirichlet Laplacian in the space of complex-valued square-summable functions $L^{2}(\Omega)$, and define $A_{i} = -\Delta_{\mathbf{x}_{i}}$ on the domain $D(A_{i})
      = \left\{u \in L^{2}(\Omega)\,:\,
         \Delta_{\mathbf{x}_{i}}u \in L^{2}(\Omega)\right\} $, with respect to the variables
$\mathbf{x}_{i} = (x_{1},\cdots,x_{d_{i}}) \in \Omega_{i}$. Thus the operators $A_{i}$ are self-adjoint and nonnegative, with discrete spectrum
$$
\left\{ \sigma_{i,k_{i}}\right\}_{k_{i} \in \N} \subset \mathbb{R}_{+}
$$
consisting in monotone increasing sequences of nonnegative numbers
$$
   0 \leq \sigma_{i,1} \leq  \cdots \leq \sigma_{i,k_{i}-1} \leq \sigma_{i,k_{i}} \to \infty
   \quad
   \text{as}
   \quad
   k_{i} \to \infty,
$$
to be associated to the smooth product eigenfunctions
$v_{k_{1}}(\mathbf{x}_{1}) \cdots v_{k_{n}}(\mathbf{x}_{n})$, respectively. For example, in the special case when
$$
d_{1} = \cdots = d_{n}=1
\e
\Omega_{1} = \cdots = \Omega_{n} = (0,\pi),
$$
as it is well known, one has
$$
\sigma_{k_{i}} = k_{i}^{2},
\e
v_{k_{1}}(x_{1}) \cdots v_{k_{n}}(x_{n}) = \frac{2^{n}}{\pi^{n}}\sin(k_{1}x_{1}) \cdots\sin(k_{n}x_{n}), 
$$
for
$(x_{1},\cdots, x_{n}) \in (0,\pi)^{n}$. 

Here, we stress that
$$
  D(A_{i}^{1/2})
= \left\{u \in L^{2}(\Omega)\,:\,
  \nabla_{\mathbf{x}_{i}}u \in L^{2}(\Omega)^{d_{i}},
  u(\cdot,\mathbf{x}_{i},\cdot) \equiv 0 \text{ on } \partial\Omega_{i}\right\},
$$
with
 $\|A_{i}^{1/2} u \|_{L^{2}(\Omega)}
= \|\nabla_{\mathbf{x}_{i}} u \|_{L^{2}(\Omega)^{d_{i}}}$,
having denoted by
$\nabla_{\mathbf{x}_{i}} = (\partial_{x_{1}},\cdots,\partial_{x_{d_{i}}})$
the gradient operators with respect to $\mathbf{x}_{i}$.
By the topological equality
$$
\partial (\Omega_{1}\times \cdots \times\Omega_{n})
=
(\partial \Omega_{1} \times \cdots \times\Omega_{n}) \cup \cdots \cup (\Omega_{1} \times \cdots \times \partial\Omega_{n})
$$
it is easy to check that
$$
\displaystyle \bigcap_{i = 1}^{n} D(A_{i}^{1/2})  = H_{0}^{1}(\Omega)
\e
D=\displaystyle \bigcap_{i = 1}^{n}D(A_{i})  = H^{2}(\Omega) \cap H_{0}^{1}(\Omega).
$$
Furthermore,
they are commuting (in the sense of distributions).

Following the functional setting introduced above,
the inverse problem we want to study reads as follows.

\smallskip

\noindent
{\bf Inverse Diffusion Problem.}\;
{\it Given $\alpha\in (0,1)$ and $T>0$, find the function
$u:\Omega \times [0,T] \to \mathbb{C}$ and the positive constants
$\lambda_{1} ,\cdots, \lambda_{n}$ that fulfill the following parabolic
initial boundary value problem:
\begin{equation}
\nonumber%
\begin{cases}
      \Dta u - \lambda_{1}\Delta_{\mathbf{x}_{1}} u - \cdots - \lambda_{n}\Delta_{\mathbf{x}_{n}} u = 0
\quad  &\text{in }\Omega \times (0,T),\\[3mm]
     u(\mathbf{x}_{1},\cdots,\mathbf{x}_{n},0) = u_0(\mathbf{x}_{1},\cdots,\mathbf{x}_{n}) &  \text{ in }\Omega,\\[3mm]
     u(\mathbf{x}_{1},\cdots,\mathbf{x}_{n},t) = 0 &\text{on } \partial\Omega \times (0,T),
\end{cases}%
\end{equation}
and the overdeterminating conditions
\begin{equation}\label{ac1}
\|\nabla_{\mathbf{x}_{1}}u(\cdot,\bar T)\|_{L^{2}(\Omega)^{d_{1}}}^{2}
   = \varphi_{1},\, \dots,\,
\|\nabla_{\mathbf{x}_{n}}u(\cdot,\bar T)\|_{L^{2}(\Omega)^{d_{n}}}^{2}
   = \varphi_{n},
\end{equation}
where $\bar T$, $u_{0}$ and $\boldsymbol{\varphi}=(\varphi_{1},\cdots,\varphi_{n})$ are given.}

\smallskip

Then, Theorem
\ref{mainresult} 
yields the following existence and uniqueness result.
\begin{theorem}
Fixed $\alpha\in (0,1)$ and $T>0$, let $u_{0} \in\mathcal{A}$ be admissible, and let $\boldsymbol{\varphi}=(\varphi_{1},\cdots,\varphi_{n})$ be compatible in the sense of Definition \ref{def compat}.
Then there exists a unique  solution $(u,\lambda_{1},\cdots,\lambda_{n})$,  to the Inverse Diffusion Problem, with
$$
\begin{cases}
u\in C([T_1,T_2], D)\cap C([0,T),L^2(\Omega)),\; \Dta u\in C([T_1,T_2], L^2(\Omega)),\; 0<T_1\leq T_2<T,
\\[2mm]
(\lambda_{1},\cdots,\lambda_{n}) \in \mathbb{R}_{+}^{n},
\end{cases}
$$
where $\boldsymbol{\varphi}$ satisfies \eqref{ac1} with $\bar T\in[T_1,T_2]$.
\end{theorem}

\subsection{Identification of Lam\'e coefficients in a plane elasticity problem}\label{lame}
Let us consider a bounded domain $\Omega \subset \R^{2}$ with smooth boundary $\partial\Omega$, which is occupied by an elastic homogeneous plate. We denote by $\boldsymbol{u} = \left(u_{1},u_{2}\right)^{\mathtt{T}}$ the displacement vector, by $\nabla = \left( \partial_{x},\partial_{y} \right)^{\mathtt{T}}$ the transposed gradient operator, and by $\Delta = \left(\partial_{x}^{2} + \partial_{y}^{2},\partial_{x}^{2} + \partial_{y}^{2}\right)^{\mathtt{T}}$ the 2-dimensional Laplace operator. Accordingly, we set the vector notations
$$
\boldsymbol{L}^{2}(\Omega) = L^{2}(\Omega) \times L^{2}(\Omega),
\,
\boldsymbol{H}^{1}(\Omega) = H^{1}(\Omega) \times H^{1}(\Omega)
\text{ and }
\boldsymbol{H}^{2}(\Omega) = H^{2}(\Omega) \times H^{2}(\Omega).
$$
We shall suppose that the plate is \emph{clamped} on $\Gamma_{D} \subset \partial \Omega$, i.e.,
$$
\boldsymbol{u} \equiv \boldsymbol{0}
\quad
\text{on}
\quad
 \Gamma_{D}
\quad
\text{\it (Dirichlet homogeneous conditions)}.
$$
and that it is \emph{stress-free} on $\Gamma_{N} \subset \partial \Omega$, i.e.,
$$
\partial_{\nu}\boldsymbol{u} \equiv \boldsymbol{0}
\quad
\text{on}
\quad
\Gamma_{N}
\quad
\text{\it (Neumann homogeneous conditions)},
$$
where $\partial_{\nu}$ is the outer normal derivative on $\Gamma_{N}$. Here we assume $\partial \Omega = \Gamma_{D} \cup \Gamma_{N}$ and $\Gamma_{D} \cap \Gamma_{N} = \emptyset$.

Following Hooke's law, the planar stress tensor is given by
$$
\boldsymbol{\tau}_{\lambda,\mu}({\boldsymbol{u}}) = \lambda \mathrm{Tr}\left( \boldsymbol{\varepsilon}(\boldsymbol{u}) \right)I + 2\mu \boldsymbol{\varepsilon}(\boldsymbol{u}),
$$
where we use the standard notations
$$
\boldsymbol{\varepsilon}(\boldsymbol{u})
=
\frac{1}{2}\left( \nabla \boldsymbol{u} + \nabla \boldsymbol{u}^{\mathtt{T}}\right)
=
\frac{1}{2}
\begin{bmatrix}
2\partial_{x}u_{1} &  \partial_{y}u_{1}  +   \partial_{x}u_{2}
\\[2mm]
\partial_{y}u_{1}  +   \partial_{x}u_{2} &  2\partial_{y}u_{2}
\end{bmatrix}
\quad
(\text{\it strain tensor})
$$
and
$$
\mathrm{Tr}\left( \boldsymbol{\varepsilon}(\boldsymbol{u}) \right)I
=
\begin{bmatrix}
\partial_{x}u_{1} +  \partial_{y}u_{2} & 0
\\[2mm]
0 & \partial_{x}u_{1} +  \partial_{y}u_{2}
\end{bmatrix}
\quad
(\text{\it trace tensor}).
$$
The parameters $\lambda$ and $\mu$ are the so-called {\it Lam\'e coefficients}, which, from now on, are to be considered independent on time and space. Since
$$
\nabla \cdot
\mathrm{Tr}\left( \boldsymbol{\varepsilon}(\boldsymbol{u}) \right)I
=
\begin{bmatrix}
\partial_{x}^{2}u_{1} +  \partial_{x}\partial_{y}u_{2}
\\[2mm]
\partial_{x}\partial_{y} u_{1} +  \partial_{y}^{2}u_{2}
\end{bmatrix}
\e
2\nabla \cdot \boldsymbol{\varepsilon}(\boldsymbol{u}) =
\begin{bmatrix}
2\partial_{x}^{2}u_{1} + \partial_{y}^{2}u_{1}  +   \partial_{x}\partial_{y}u_{2}
\\[2mm]
\partial_{x}\partial_{y}u_{1}  +   \partial_{x}^{2}u_{2} + 2\partial_{y}^{2}u_{2}
\end{bmatrix},
$$
then the second-order operator $-\nabla \cdot \boldsymbol{\tau}_{\lambda,\mu}({\boldsymbol{u}})$ can be split as
\begin{equation}\label{decomposition}
-\nabla \cdot \boldsymbol{\tau}_{\lambda,\mu}({\boldsymbol{u}}) = (\lambda + \mu) \boldsymbol{A}\boldsymbol{u} + \mu \boldsymbol{B}\boldsymbol{u}
\end{equation}
where we define
$$
\boldsymbol{A}\boldsymbol{u} =
-
\nabla\left( \nabla \cdot \boldsymbol{u} \right)
=
-
\begin{bmatrix}
\partial_{x}^{2}u_{1} +  \partial_{x}\partial_{y}u_{2}
\\[2mm]
\partial_{x}\partial_{y} u_{1} +  \partial_{y}^{2}u_{2}
\end{bmatrix}
$$
on the domain
$$
D(\boldsymbol{A}) = \left\{ \boldsymbol{u} \in \boldsymbol{L}^{2}(\Omega); \boldsymbol{A}\boldsymbol{u} \in \boldsymbol{L}^{2}(\Omega) \text{ and } \boldsymbol{u} \equiv \boldsymbol{0} \text{ on }\Gamma_{D}\right\}
$$
and
$$
\boldsymbol{B}\boldsymbol{u} = - \Delta \boldsymbol{u}
=
-
\begin{bmatrix}
\partial_{x}^{2}u_{1} +  \partial_{y}^{2}u_{1}
\\[2mm]
\partial_{x}^{2}u_{2} +  \partial_{y}^{2}u_{2}
\end{bmatrix}
$$
on the domain
$$
D(\boldsymbol{B}) = \left\{ \boldsymbol{u} \in \boldsymbol{L}^{2}(\Omega); \boldsymbol{B}\boldsymbol{u} \in \boldsymbol{L}^{2}(\Omega) \text{ and } \boldsymbol{u} \equiv \boldsymbol{0} \text{ on }\Gamma_{D} \right\}.
$$
By straightforward computations, it is not difficult to see that both operators $\boldsymbol{A}$ and $\boldsymbol{B}$ are nonnegative and self-adjoint, and mutually commute. Then, following the general theory, we can define $\boldsymbol{A}^{1/2}$ and $\boldsymbol{B}^{1/2}$ on the respective domains $D(\boldsymbol{A}^{1/2})$ and $D(\boldsymbol{B}^{1/2})$, with
$$
 \int_{\Omega}|\boldsymbol{A}^{1/2}\boldsymbol{u}|^{2} \, dxdy
=
\int_{\Omega} \boldsymbol{A}\boldsymbol{u}\cdot \overline{\boldsymbol{u}} \,\de x\,\de y
=
\int_{\Omega} |\nabla \cdot \boldsymbol{u}|^{2}\,\de x\,\de y
$$
and
$$
\int_{\Omega}|\boldsymbol{B}^{1/2}\boldsymbol{u}|^{2} \, \de x\,\de y
=
\int_{\Omega} \boldsymbol{B}\boldsymbol{u}\cdot \overline{\boldsymbol{u}} \,\de x\,\de y
=
\int_{\Omega} |\nabla \boldsymbol{u}|^{2}\,\de x\,\de y.
$$
Here, since $D(\boldsymbol{B}^{1/2}) \subseteq D(\boldsymbol{A}^{1/2})$ and $D(\boldsymbol{B}) \subseteq D(\boldsymbol{A})$, there holds
$$
D(\boldsymbol{B}^{1/2}) \cap D(\boldsymbol{A}^{1/2}) = D(\boldsymbol{B}^{1/2}) = \boldsymbol{H}_{D}^{1}(\Omega)
= \left\{ \boldsymbol{u} \in \boldsymbol{H}^{1}(\Omega); \boldsymbol{u} \equiv \boldsymbol{0} \text{ on }\Gamma_{D} \right\}
$$
and
$$
D(\boldsymbol{B}) \cap D(\boldsymbol{A}) = D(\boldsymbol{B}) = \boldsymbol{H}^{2}(\Omega) \cap \boldsymbol{H}_{D}^{1}(\Omega).
$$

Then, we can state our inverse problem consisting in the identification of the Lam\'e coefficients, along with the displacement function $\boldsymbol{u}$, as in the following.

\smallskip


\noindent
{\bf Inverse Lam\'e Problem.}\;
{\it Given $\alpha\in (0,1)$ and $T>0$, find the function
$\boldsymbol{u}:\Omega \times [0,T] \to \mathbb{C}^{2}$ and the constants
$\lambda$ and $\mu$ that fulfill the parabolic
initial boundary value problem:
\begin{equation}
\nonumber
\begin{cases}
      \Dta\boldsymbol{u} -\nabla \cdot \boldsymbol{\tau}_{\lambda,\mu}({\boldsymbol{u}}) = \boldsymbol{0}
  &\text{in }\Omega \times (0,T),
\\[3mm]
     \boldsymbol{u}(x,y,0) = \boldsymbol{u}_0(x,y) &\text{in }\Omega,
\\[3mm]
     \boldsymbol{u}(x,y,t) = \boldsymbol{0}
  &\text{on }\Gamma_{D} \times (0,T),
\\[3mm]
\partial_{\nu}\boldsymbol{u}(x,y,t) = \boldsymbol{0}
&\text{on } \Gamma_{N} \times (0,T),
\end{cases}%
\end{equation}
and the overdeterminating conditions
\begin{equation}\label{ac2}
\int_{\Omega} |\nabla \cdot \boldsymbol{u}(x,y,\bar T)|^{2}\,\de x\,\de y
 = \varphi \e
\int_{\Omega} |\nabla \boldsymbol{u}(x,y,\bar T)|^{2}\,\de x\,\de y  = \psi,
\end{equation}
where $\bar T$, $\boldsymbol{u}_{0}$, $\varphi$ and $\psi$ are given.}

\smallskip

In order to deduce our existence and uniqueness result, we now apply Theorem \ref{mainresult}
where, 
according to the decomposition \eqref{decomposition}, we set the positive unknown constants as
$$
\lambda_1 := \lambda + \mu
\e
\lambda_2 := \mu.
$$
Notice that the condition $\lambda_1 > 0$ is actually the strong convexity condition in dimension two (cf. \cite{ANS,NU}). With these choices, we get the following result.
\begin{theorem}
Given $\alpha\in (0,1)$ and $T>0$, let $\boldsymbol{u}_{0} \in \mathcal{A}$ be admissible, and let $(\varphi,\psi)$ be compatible in the sense of Definition \ref{def compat}.
Then there exists a unique  solution $(\boldsymbol{u},\lambda,\mu)$  to the Inverse Lam\'e Problem $(\boldsymbol{u},\lambda,\mu)$, with
$$
\begin{cases}
\boldsymbol{u} \in 
C([0,T);\boldsymbol{L}^{2}(\Omega)) \cap C([T_1,T_2];D(\boldsymbol{B})),\; \Dta\boldsymbol{u}\in C([T_1,T_2], \boldsymbol{L}^2(\Omega)),\; 0<T_1\leq T_2<T,
\\[2mm]
\lambda \in \R, \quad \mu \in \R_{+} 
\e
\lambda + \mu >0,
\end{cases}
$$
where $\boldsymbol{\varphi}$ satisfies \eqref{ac2} with $\bar T\in[T_1,T_2]$.
\end{theorem}

\begin{remark} 
We stress that the physical meaning of the Inverse Lam\'e Problem, as it is here
stated, represents the in-plane motion of an elastic flat plate of
negligible mass where the inertial term $\partial_{t}^{2} \boldsymbol{u}$ in the dynamic equilibrium equation has been neglected, as the parabolic abstract theory we developed cannot manage directly a second-order equation in time (cf. \cite[Remark 3]{ Mola1}).
\end{remark}

\section*{Acknowledgments}\small
\noindent S. C., M.R. L., G. M. and S. R. have been supported by the Gruppo Nazionale per l'Analisi Matematica, la Probabilit\`a e le loro Applicazioni (GNAMPA) of the Istituto Nazionale di Alta Matematica (INdAM). A. M. has been supported by the Gruppo Nazionale per il Calcolo Scientifico (GNCS) of the Istituto Nazionale di Alta Matematica (INdAM).\\
S. C. and M.R. L. also thank MUR for the support under the project PRIN 2022 -- 2022XZSAFN: \lq\lq Anomalous Phenomena on Regular and Irregular Domains: Approximating Complexity for the Applied Sciences" -- CUP B53D23009540006 (see the website\\ \href{https://www.sbai.uniroma1.it/~mirko.dovidio/prinSite/index.html}{https://www.sbai.uniroma1.it/~mirko.dovidio/prinSite/index.html}).\\

\appendix

\section{Numerical simulations: additional discussion}\label{numerical}

In the present section, we provide more specific comments to the plots presented in Fig.~\ref{fig::all_results} (see Section 5). In all plots produced it is observed that, given their generally low dissipation, the simulations carried out with $\alpha=1$ appear to be less affected by changes in the diffusion parameters. Factor 10 variations in the input parameters values result in fact in much smaller variation of the output parameter values. On the other hand, the simulations featuring $\alpha=0.75$ and even more so the ones with $\alpha=0.25$ show substantial variations of solution derivative norms associated to variations of the input parameters. Despite the fact that absolute output parameters variations appear wider for the intermediate $\alpha=0.75$ simulations, the relative output swings are considerably higher in the $\alpha=0.25$ test cases. In fact, for such a value of $\alpha$, the output parameters show variations of a factor 10 as the diffusion parameter values change. 

\begin{figure}[!h]
	\includegraphics[width=11cm]{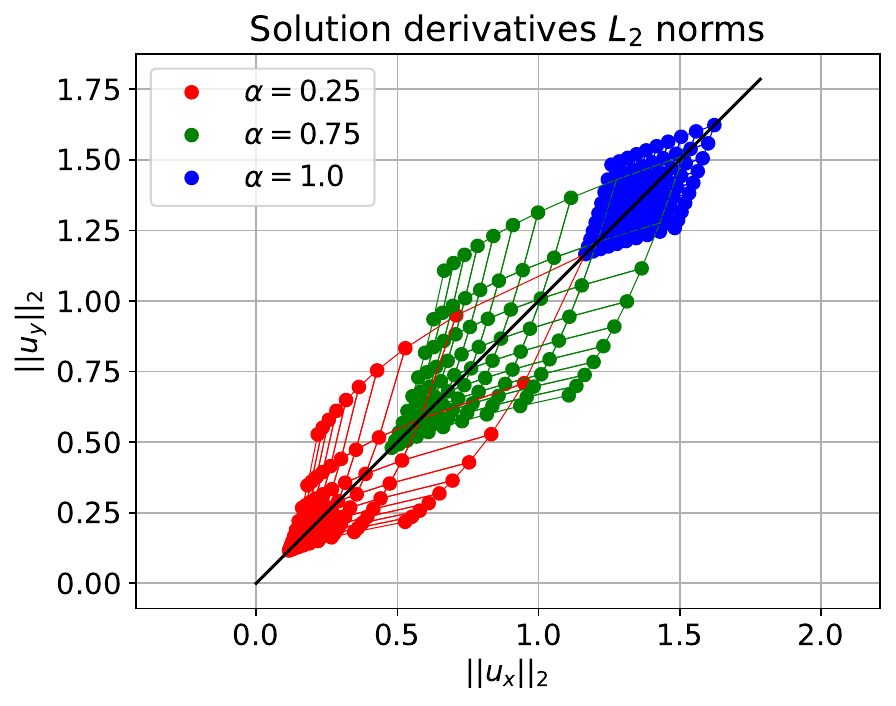}
	\caption{Solution derivatives norms $||u_x||_2$ vs $||u_y||_2$ obtained using the diffusion parameters in Fig~\ref{fig::diff_params} in the problem with L-shaped domain and initial condition featuring a single symmetric $C^\infty$ cutoff function. \label{fig::output_L_symm}}
\end{figure}

\begin{figure}[!h]
	\includegraphics[width=11cm]{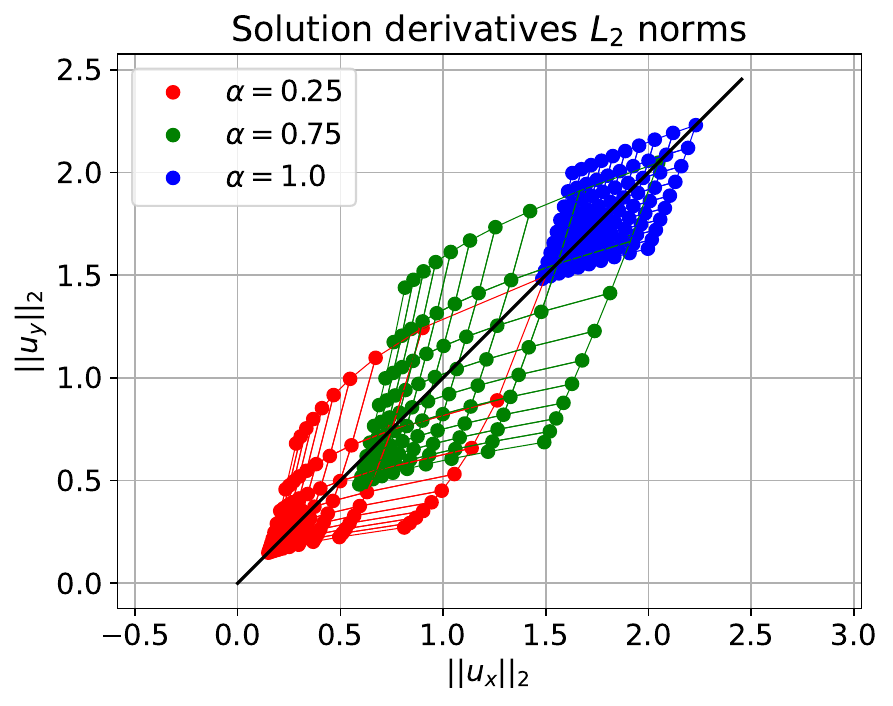}
	\caption{Solution derivatives norms $||u_x||_2$ vs $||u_y||_2$ obtained using the diffusion parameters in Fig~\ref{fig::diff_params} in the problem with a L-shaped domain and initial condition featuring two asymmetrically centered $C^\infty$ cutoff functions. \label{fig::output_L_asymm}}
\end{figure}

Fig. \ref{fig::output_L_symm} refers to the map $\boldsymbol{\F}_{\alpha}$ obtained through the numerical simulations carried out using the L-shaped domain $\Omega_1$ and the initial condition $\phi_1^{\Omega_1}$ characterized by a single $C^\infty$ cutoff function placed in a symmetric location. As can be appreciated in the image, in such a symmetric configuration all the simulations in which $\lambda_1=\lambda_2$ resulted in equal values of the output $||u_x||_2= ||u_y||_2$, and, in general, the plot appears symmetric with respect to the first quadrant angle bisector $y=x$. In particular, in all the simulations where $\lambda_1=\lambda_2$ we have $||u_x||_2= ||u_y||_2$. It must be remarked that the observed symmetry is a confirmation of validity of the resolution algorithm employed, as the numerical error is not able to spoil the symmetric behavior expected in the present test case.

\begin{figure}[!h]
	\includegraphics[width=11cm]{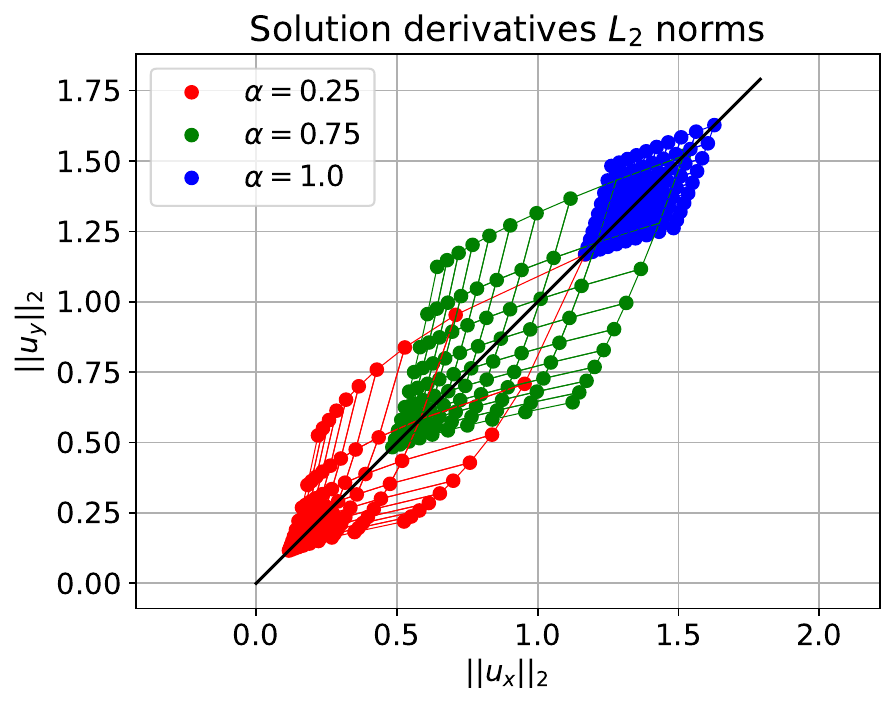}
	\caption{Solution derivatives norms $||u_x||_2$ vs $||u_y||_2$ obtained using the diffusion parameters in Fig~\ref{fig::diff_params} in the problem with a circular domain with eccentric hole and initial condition featuring a single symmetric $C^\infty$ cutoff function. \label{fig::output_eccentric_symm}}
\end{figure}

\begin{figure}[!h]
	\includegraphics[width=11cm]{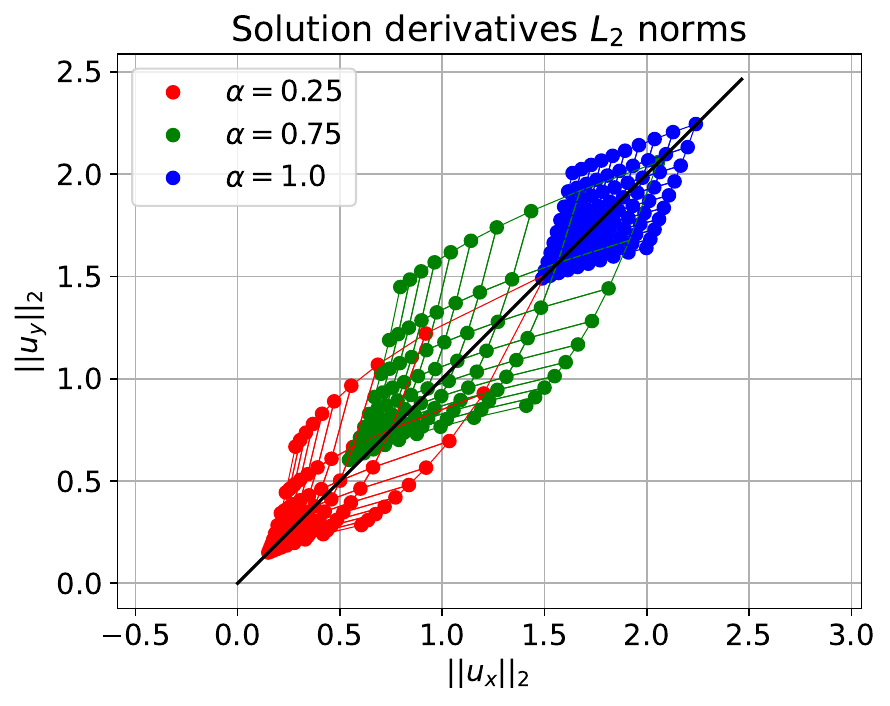}
	\caption{Solution derivatives norms $||u_x||_2$ vs $||u_y||_2$ obtained using the diffusion parameters in Fig~\ref{fig::diff_params} in the problem with a circular domain with eccentric hole and initial condition featuring two asymmetrically centered $C^\infty$ cutoff functions. \label{fig::output_eccentric_asymm}}
\end{figure}

Fig. \ref{fig::output_L_asymm} depicts the numerical approximation of map $\boldsymbol{\F}_{\alpha}$ obtained with the L-shaped domain $\Omega_1$ and initial condition $\phi_2^{\Omega_1}$ characterized by the sum of two $C^\infty$ cutoff functions placed in non symmetric fashion. As expected, here the output parameters plot does not appear symmetric with respect to the first quadrant angle bisector $y=x$. This is particularly clear at the corners of the $\alpha=0.25$ and $\alpha=0.75$ grids which are located farthest from the first quadrant angle bisector, and in the left bottom corner of the green grid associated with $\alpha=0.75$. A similar asymmetry is also observed for the red grid corresponding to $\alpha=0.25$, but given the smaller output parameter values associated to such plots, the asymmetric behavior is not fully readable with the scale considered. 

Fig. \ref{fig::output_eccentric_symm} presents an illustration of map $\boldsymbol{\F}_{\alpha}$ obtained using the circular domain with eccentric hole $\Omega_2$ and the initial condition $\phi_1^{\Omega_2}$ characterized by a single $C^\infty$ cutoff function placed in a symmetric position. Also in this case, the results diagram appears symmetric, as is to be expected given the symmetric configuration of domain and initial conditions. With this domain the less dissipative simulations remain the ones associated with $\alpha=1$, while the solution derivatives norms present progressively decreasing values when $\alpha=0.75$ $\alpha=0.25$, due to the increased dissipation. This is also confirmed by the results obtained using  the circular domain with eccentric hole $\Omega_2$ and the initial condition $\phi_2^{\Omega_2}$ characterized by the sum of two $C^\infty$ cutoff functions placed in non symmetric positions. As expected, the corresponding plots, presented in Fig.~\ref{fig::output_eccentric_asymm}, lose symmetry with respect to the first quadrant angle bisector. Once again, this is, in particular, visible at the bottom left corners of each output grid, corresponding to the most dissipative simulations carried out. And, as was the case with the L-shaped domain, also with the present circular domain with eccentric hole, the most evident departures from a symmetric behavior are observed for the $\alpha = 0.75$ simulations, for which the entire lower end of the output plot is placed at a reasonable distance from the first quadrant angle bisector. 

\bigskip

\end{document}